\documentclass{amsart}
\usepackage{verbatim,amssymb,amsmath,amscd,latexsym,amsbsy,mathrsfs}
\usepackage{enumerate}

\usepackage[pdfencoding=auto,psdextra]{hyperref}
\usepackage{bookmark}

\input{xy}
\xyoption{all}

\newtheorem{thm}{Theorem}[section] \newtheorem{pro}[thm]{Proposition}
\newtheorem{lemma}[thm]{Lemma}
\newtheorem{cor}[thm]{Corollary}
\numberwithin{equation}{section}

\newtheorem{question}[thm]{Question}

\theoremstyle{definition} \newtheorem*{ex}{Example}
\newtheorem{rmk}[thm]{Remark} 

\newtheorem*{notation}{Notation}

\newtheorem{df}[thm]{Definition}
\newtheorem{hypothesis}[thm]{Hypothesis}

\DeclareMathAlphabet{\mathpzc}{OT1}{pzc}{m}{it}
\DeclareMathOperator*{\Spec}{Spec}\DeclareMathOperator*{\spf}{Spf}
\DeclareMathOperator*{\spec}{Spec} \DeclareMathOperator*{\Hom}{Hom}
 
\DeclareMathOperator*{\Aut}{Aut} \DeclareMathOperator*{\QF}{QF}
\DeclareMathOperator*{\Gal}{Gal}

\DeclareMathOperator*{\ord}{ord}
\DeclareMathOperator*{\supp}{Supp}
\DeclareMathOperator*{\BL}{BL}
\DeclareMathOperator*{\Sh}{Sh}
\DeclareMathOperator*{\GL}{GL}
\DeclareMathOperator*{\Cov}{Cov}

\DeclareMathOperator*{\into}{\hookrightarrow}
 \newcommand{\NN}{\mathbb{N}}
 \newcommand{\QQ}{\mathbb{Q}}
\newcommand{\ZZ}{\mathbb{Z}} 
\newcommand{\PP}{\mathbb{P}}

\newcommand{\cO}{\mathcal{O}} 
\newcommand{\scrZ}{\mathcal{Z}}
\newcommand{\scrX}{\mathcal{X}}\newcommand{\scrY}{\mathcal{Y}}
\newcommand{\scrW}{\mathcal{W}}
\newcommand{\scrF}{\mathcal{F}}\newcommand{\scrG}{\mathcal{G}}
\newcommand{\scrV}{\mathcal{V}}\newcommand{\scrU}{\mathcal{U}}
\newcommand{\scrA}{\mathcal{A}}

\newcommand{\cU}{\mathscr{U}}
\newcommand{\cV}{\mathscr{V}} 
\newcommand{\scrp}{\mathfrak{p}} 
\newcommand{\scrq}{\mathfrak{q}} 
\newcommand{\scrm}{\mathfrak{m}}

\newcommand{\cK}{\mathcal{K}}

 \newcommand{\Q}{\mathbb{Q}}
\newcommand{\Z}{\mathbb{Z}}

\begin{document}
\title{Ramification theory and formal orbifolds in arbitrary dimension} 
 \author{
  Manish Kumar
 }
 \address{
Statistics and Mathematics Unit\\
Indian Statistical Institute, \\
Bangalore, India-560059
  }
\email{manish@isibang.ac.in}
 \begin{abstract}
   Formal orbifolds are defined in higher dimension. Their \'etale fundamental groups are also defined. It is shown that the fundamental groups of formal orbifolds have certain finiteness property and it is also shown that they can be used to approximate the \'etale fundamental groups of normal varieties. Etale site on formal orbifolds are also defined. This framework allows one to study wild ramification in an organised way. Brylinski-Kato filtration and $l$-adic sheaves in these contexts are also studied.   
  \end{abstract}
\maketitle

\section{Introduction}

Let $X$ be a normal variety. Tamely ramified covers of $X$ are easy to understand. This is because locally these covers are determined by the ramification index. But this is not at all true for wildly ramified covers of $X$. This makes the study of wild ramification very tedious. In this paper we define a structure $P$ on $X$ called the branch data (see Section \ref{sec:branch-data}). When $X$ is a curve over an algebraically closed field then these objects were defined and studied in \cite{orbifold.bundles}. For a closed point $x\in X$, $P(x)$ is just a finite Galois extension of the fraction field of $\widehat \cO_{X,x}$ like in \cite{orbifold.bundles} but for higher dimensional $X$ the definition of branch data is slightly more intricate (see Definition \ref{def:branch-data}). The pair $(X,P)$ is called a formal orbifold. These objects provide a framework to study wild ramification in a relatively organised way.

Section \ref{sec:local-theory} develops ramification theory for excellent normal rings $R$ of arbitrary dimension. Some variants of which can be found in \cite{abhyankar}. Lemma \ref{compositum-wild-ramification} is the main result of this section which generalizes \cite[Lemma 3.1]{killwildram} to higher dimension. In section \ref{sec:branch-data} branch data on $X$ is defined and some of its basic properties are studied. Every finite dominant separable morphism from $Y\to X$ give rise to a Branch data on $X$ (Proposition \ref{B_f}). Formal orbifolds are defined in section \ref{sec:formal-orbifolds} to be normal varieties $X$ together with a branch data $P$. Coverings, \'etale and flat covers of formal orbifolds are also studied in this section. Fiber products of coverings and their properties are also discussed. It is shown that fiber product of \'etale coverings of formal orbifolds is an \'etale covering (Corollary \ref{fiber-product-etale}).

The \'etale fundamental group of formal orbifolds are defined in section \ref{sec:etale-pi_1}. Their basic properties which are analogous to fundamental groups of schemes are also proved. The following nice property about them are proved.
\begin{thm}(see Corollary \ref{pi_1-finite-rank})
 Let $(X,P)$ be a connected formal orbifold with $X$ projective over $\spec(k)$ where the field $k$ is such that the absolute Galois group $G_k$ has finite rank then $\pi_1(X,P)$ has finite rank.
\end{thm}
It is also shown that they approximate the fundamental group of a normal variety.
\begin{thm}(see Theorem \ref{approximation})
 Let $X$ be a normal connected variety over a perfect field or $\spec(\ZZ)$ and $U$ an open subset of $X$. Then $\pi_1(U)\cong \varprojlim\pi_1(X,P)$ where $P$ runs over all the branch data on $X$ such that $\BL(P)\subset X\setminus U$. 
\end{thm}
The above two results applied together implies that $\pi_1(U)$ for $U$ quasiprojective variety over $k$ is the inverse limit of finite rank profinite groups which have a geometric interpretation.


In section \ref{sec:site-sheaves} \'etale site, presheaves and sheaves on formal orbifolds are defined. Structure sheaf of formal orbifolds are defined and their stalks have been computed. Note that a formal orbifold $(X,O)$ where $O$ is the trivial branch data is essentially the same as the normal variety $X$. In particular the \'etale site on $(X,O)$ is the same as the small \'etale site on $X$ (and hence $\pi_1(X,O)=\pi_1(X)$). A geometric formal orbifold is a formal orbifold $\scrX=(X,P)$ which admits a finite Galois \'etale morphism $f:(Y,O)\to (X,P)$ (in the sense of formal orbifold) for some $Y$ with trivial branch data $O$. In such a situation cohomology groups of $\scrX$ can be computed using the cohomology groups of $Y$ by a spectral sequence (Proposition \ref{spectral-sequence}).

In section \ref{sec:BK-filtration} using swan conductors of local fields defined by Kato and Matsuda a divisor $Sw(X,P)$ is associated to branch data $P$ on $X$ whose branch locus has irreducible components of pure codimension one in $X$. The abelianized fundamental group parameterizing covers with certain bounded ramification $\pi_1^{ab}(X,D)$ defined by Kerz and Saito in \cite{Kerz-Saito.CLT} and \cite{Kerz-Saito.Lefschetz} for an effective Cartier divisor $D$ on $X$ is also shown to be the inverse limit of $\pi_1^{ab}(X,P)$ where $P$ varies over branch data with $Sw(X,P)\le D$. This leads to questions such as whether some version of  class field theory holds for formal orbifolds (\ref{Q:CFT}) and whether a version of Lefschetz's theorem for fundamental groups is true for formal orbifolds (\ref{Q:Lefschetz}). 

In section \ref{sec:Lefschetz} we make some progress towards the later. Note that Lefschetz theorem for fundamental group is false for affine varieties in positive characteristic. Though Lefschetz theorem holds if one restricts to prime-to-p \'etale fundamental group or more generally tame fundamental group (\cite{Esnault-Kindler}). To allow wild ramification as well, one has to bound the ramification in some sense. In \cite{Kerz-Saito.Lefschetz} Kerz and Saito prove Lefschetz theorem for $\pi_1^{ab}(X,D)$. The definition of $\pi_1^{ab}(X,D)$ is the dual of the first cohomology group and their proof is cohomological in nature. Here we follow the method adopted in \cite[Chapter XII]{SGA2}, \cite[4.3]{Grothendieck-Murre} and \cite{Esnault-Kindler}. More precisely, we show the following results.

\begin{pro}(Proposition \ref{formal-fullyfaithful})
 Let $(X,P)$ be a projective formal orbifold over a perfect field $k$ of dimension at least two, $Y$ be a normal hypersurface of $X$ which intersects $\BL(P)$ transversally and the divisor $[Y]$ is a very ample divisor on $X$. Then the functor $\Cov(X,P)\to \Cov(\widehat X_Y,\hat P)$ is fully faithful. Here $\widehat X_Y$ is the completion $X$ along $Y$ and $\hat P$ is the branch data on $\widehat X_Y$ induced from $P$.
\end{pro}

\begin{thm}(Theorem \ref{GM-wild})
 Under the Hypothesis \ref{nonsplit}, the functor $\Cov(\widehat{X}_Y,\hat P) \to \Cov(Y,P_{|Y})$ is an equivalence.
\end{thm}

The Hypothesis \ref{nonsplit} essentially says that the formal orbifold $(X,P)$ is geometric, irreducible, smooth and projective over $k$; $Y$ is a smooth irreducible hypersurface of $X$ not contained in $\BL(P)$ such that for any point $y\in \BL(P)\cap Y$ there is only one point lying above it in the cover obtained from taking the normalization of a formal local neighbourhood $\widehat U_y$ of $y$ in $P(y,U)$. As a consequence of the above two results we get the following:

\begin{cor} (Corollary \ref{Lefschetz-thm})
  Let $(X,P)$ be a smooth irreducible projective formal orbifold over a perfect field $k$ of dimension at least two. The functor from $\Cov(X,P)\to \Cov(Y,P_{|Y})$ is fully faithful if $Y$ satisfies Hypothesis \ref{nonsplit}. In particular the map natural map $\pi_1(Y,P_{|Y})\to \pi_1(X,P)$ is an epimorphism. 
\end{cor}

The above is a wild analogue of \cite[4.3.7]{Grothendieck-Murre}.

In section \ref{sec:l-adic} $l$-adic sheaves on $X^o$ with normal compactification $X$ is analyzed in the language of formal orbifolds. It is shown that a locally constant sheaf $\scrF^o$ on $X^o$ extends to a locally constant sheaf $\scrF$ on a geometric formal orbifold $(X,P)$ for some branch data $P$ with branch locus contained in the boundary. It is also shown that if the ramification of $\scrF^o$ is bounded by a Galois \'etale cover $f:Y^o\to X^o$ (in the sense of \cite{Drinfeld}, \cite{Esnault-Kerz.Deligne}, \cite{Esnault}, etc.) then one could choose $P$ such that the branch data $f^*P$ on the normalization $Y$ of $X$ in $k(Y^o)$ is curve-tame. This allows us to observe that a lisse $l$-adic sheaf on $X^o$ has ramification bounded by a morphism $f$ then it has ramification bounded by any other morphism whose branch data is same as the branch data associated to $f$ up to tame part (Corollary \ref{lisse-bounded-ramification}).  

\section*{Acknowledgements}
The author thanks Suresh Nayak for some useful discussions and an anonymous referee for useful suggestions.

\section{Local Theory} \label{sec:local-theory}

We begin with some results on ramification theory. Proposition \ref{basic-ramification} and \ref{unramified-closed-points} below are variants of results in Section 7 of \cite{abhyankar} where the treatment is for arbitrary normal local domains and hence the comparison with completion of the rings are missing.

For a ring $R$ and a prime ideal $P$, let $\widehat R^P$ denote the $P$-adic completion of $R$. For a local ring $R$, $\widehat R$ denotes the completion of $R$ at the maximal ideal. For an integral domain $R$, $\QF(R)$ denotes the field of fractions (=quotient field) of $R$.

We recall the definition of decomposition group and inertia group for finite extension of normal domains. \cite[page 35-37]{abhyankar}
\begin{df}
Let $R$ be an excellent normal domain with fraction field $K$, $L$ a finite $G$-Galois extension of $K$ and $S$ the integral closure of $R$ in $L$. Let $Q\subset S$ be a prime ideal of $S$ and $P=Q\cap R$. 
\begin{enumerate}
 \item Define the \emph{decomposition group} $D(Q)=\{\sigma\in G: \sigma(Q) \subset Q\}$
 \item Define the \emph{inertia group} $I(Q)=\{\sigma\in G: \sigma(\alpha)-\alpha\in Q \text{ for all } \alpha\in S \}$. Note that $I(Q)\subset D(Q)$. 
\end{enumerate}
\end{df}

The following result is well known but a precise reference is difficult to find in the literature.

\begin{pro}\label{basic-ramification}
Let $R$ be an excellent normal domain with fraction field $K$, $L$ a finite separable extension of $K$ and $S$ the integral closure of $R$ in $L$. Let $Q\subset S$ be a prime ideal of $S$ and $P=Q\cap R$. Let $\widehat K_P=\QF(\widehat R_P)$, $\widehat L_Q=\QF(\widehat S_Q)$, $\widehat K^P=\QF(\widehat R^P)$ and $\widehat L^Q=\QF(\widehat L^Q)$. Then the following holds:
\begin{enumerate}
 \item The field extensions $\widehat L_Q/\widehat K_P$ and $\widehat L^Q/\widehat K^P$ are finite separable.
 \item If $L/K$ is Galois then so are $\widehat L_Q/\widehat K_P$ and $\widehat L^Q/\widehat K^P$.
 \item If $G=\Gal(L/K)$ then $\Gal(\widehat L_Q/\widehat K_P)$ is isomorphic to the subgroup $D(Q)$ and $\Gal(\widehat L_Q/\widehat K_P)\le \Gal(\widehat L^Q/K^P)$.
 \item Moreover, if $Q'$ is another prime ideal of $S$ lying above $P$ then $D(Q)$ and $D(Q')$ are conjugates in $G$. 
\end{enumerate}
\end{pro}

\begin{proof}
  Since $R_P$ and $S_Q$ are excellent normal local domains, $\widehat R_P$ and $\widehat S_Q$ are also complete local normal domains containing $R_P$ and $S_Q$ respectively (\cite[7.8.3(v)]{ega4.2}). Also note that $\widehat R^P \subset \widehat R_P$ and $\widehat S^Q \subset \widehat S_Q$. 
 Hence $L$, $\widehat K^P$ and $\widehat K_P$ are subfields of $\widehat L_Q$, $\widehat L^Q=\widehat K^PL$ and $\widehat L_Q=\widehat K_P L$. So $\widehat L_Q/\widehat K_P$ is separable because $L/K$ is separable and $K\subset \widehat K_P$. The same proof says if $L/K$ is a normal extension then so is $\widehat L_Q/\widehat K_P$. Hence if $L/K$ is Galois then so is $\widehat L_Q/\widehat K_P$.
 
 By definition of $D(Q)$, if $a\in Q^n$ then $ga\in Q^n$ for every $g\in D(Q)$. Hence given a Cauchy sequence $(\alpha_i)_{i\in \NN}\in L$ and $g\in D(Q)$ then sequence $(g\alpha_i)_{i\in \NN}$ is also a Cauchy sequence. In particular $D(Q)$ acts on $\widehat L_Q$. Also it fixes $\widehat K_P$. So $D(Q)\into \Gal(\widehat L_Q/\widehat K_P)$ is a monomorphism.
 
 Let $\sigma \in \Gal(\widehat L_Q/\widehat K_P)$. Note that $L/K$ is Galois and $\sigma_{|K}=id_K$. Hence $\sigma(L)=L$. Since $\widehat S_Q$ is the integral closure of $\widehat R_P$ in $\widehat L_Q$, $\sigma(\widehat S_{Q})=\widehat S_{Q}$. But $S_Q=\widehat S_Q\cap L$, so $\sigma(S_Q)=S_{Q}$. Since $\sigma$ is a ring automorphism of a local ring, $\sigma(Q)=Q$. In particular $\sigma_{|L}\in D(Q)$. Hence $D(Q) \cong \Gal(\widehat L_Q/\widehat K_P)$.
 
 To see the last statement, let $Q'$ be another prime ideal of $S$ lying above $P$. We have that $S_{Q'}\subset \widehat S_{Q'}$ and hence $L\subset \widehat L_{Q'}$. Let $M$ and $M'$ be the algebraic closure of $\widehat L_Q$ and $\widehat L_{Q'}$ then both $M$ and $M'$ are algebraic closure of $\widehat K_P$. So there exists a $\widehat K_P$-isomorphism $\phi:M\to M'$. Since $L/K$ is a Galois extension and $K\subset \widehat K_P$, $\phi(L)=L$. Note that $\widehat L_Q$ and $\widehat L_{Q'}$ are splitting fields of the polynomial $f(Z)$ over $\widehat K_P$, hence $\phi(\widehat L_Q)=\widehat L_{Q'}$. Since $\widehat S_Q$ and $\widehat S_{Q'}$ are integral closures of $\widehat R_P$ in $\widehat L_Q$ and $\widehat L_{Q'}$ respectively, $\phi(\widehat S_{Q})=\widehat S_{Q'}$. But $S_Q=\widehat S_Q\cap L$, so $\phi(S_Q)=S_{Q'}$. Hence $\phi\in \Gal(L/K)$ and $\phi(Q)=Q'$. In particular $\phi^{-1}D(Q')\phi=D(Q)$.
\end{proof}

\begin{pro}\label{decomposition-variation}
  Let $R$ be an excellent normal domain with fraction field $K$, $L$ a finite separable extension of $K$ and $S$ the integral closure of $R$ in $L$. Let $Q_2\subsetneq Q_1\subset S$ be prime ideals of $S$ and $P_i=Q_i\cap R$. Then $\widehat K^{P_2} \subset \widehat K^{P_1}$ and $\widehat L^{Q_1}=L\widehat K^{P_1}=\widehat L^{Q_2}\widehat K^{P_1}$. 
\end{pro}

\begin{proof}
 Note that $\widehat R^{P_2}$, the $P_2$-adic completion of $R$, contains $R$ since $R$ is noetherian and hence $P_2$-adically separated. Note that $P_2$-adic completion of $P_1$ is $P_1\widehat R^{P_2}$ is a prime ideal of $\widehat R^{P_2}$ and $\widehat R^{P_1}$ is also the $P_1\widehat R^{P_2}$-adic completion of $\widehat R^{P_2}$.
 Hence $\widehat R^{P_2}\subset \widehat R^{P_1}$. Hence $\widehat K^{P_2} \subset \widehat K^{P_1}$. Similarly $\widehat L^{Q_2} \subset \widehat L^{Q_1}$. Since $S$ is a finite $R$-module, $\widehat S^{Q_1}$ is a finite $\widehat R^{P_1}$-module with the same generators. Hence $\widehat L^{Q_1}=L\widehat K^{P_1}=\widehat L^{Q_2}\widehat K^{P_1}$. 
 \end{proof}

\begin{rmk}\label{inertia-variation}
  Let $Q_2\subsetneq Q_1\subset S$ be prime ideals of $S$ and $P_i=Q_i\cap R$. If $\sigma \in I(Q_2)$ then $\sigma{\alpha}-\alpha\in Q_2\subset Q_1$ for all $\alpha\in R$. Hence $I(Q_2)\subset I(Q_1)$. (Also see \cite[Proposition 1.50]{abhyankar})
\end{rmk}

\begin{rmk}\label{inertia-kernel}
  Let $Q\subset S$ be a prime ideal of $S$ and $P=Q\cap R$. Then $I(Q)$ is the kernel of the natural epimorphism from $\phi:D(Q) \to \Aut(k_Q/k_P)$ where $k_Q=\QF(S/Q)$, $k_P=\QF(R/P)$ and for $g\in D(Q)$, $a\in S_Q$ and its image $\overline a\in k_Q$, $\phi(g)(\overline a)=\overline{ga}$.
\end{rmk}

The following is a consequence of \cite[Theorem 1.47 and 1.48]{abhyankar} but we include a direct proof for brevity and convenience of the reader.
\begin{pro}\label{unramified-closed-points}
 Let $(R,m)$ be a complete local normal domain and $(S,n)$ be the integral closure $R$ in a finite Galois extension of $\QF(R)$. Then $I(n)$ is trivial iff $S/n$ is a finite separable extension of $R/m$ and $mS=n$.
\end{pro}

\begin{proof}
  Let $k_n=S/n$, $k_m=R/m$, $L=\QF(S)$ and $K=\QF(R)$. Since $L/K$ is Galois, the field extension $k_n/k_m$ is normal. If $I(n)$ is trivial then $\Gal(L/K)=D(n)\cong \Aut(k_n/k_m)$. Let $f$ be the cardinality of this group. Since $S$ is a torsion free $R$-module, the vector space dimension of $S\otimes_R k_m$ is at most $[L:K]$. Hence $\dim_{k_m}(S/mS)\le f$. Since $mS\subset n$, $$f=|Aut(k_m/k_n)|\le \dim_{k_m}(S/n)\le \dim_{k_m}(S/mS)\le f.$$ 
  Hence $k_m/k_n$ is Galois extension and $mS=n$.
  
  Conversely, by Cohen structure theorem $S=S_0[[y_1,\ldots,y_s]]$ and the subring $R=R_0[[x_1,\ldots,x_r]]$ where $R_0$ and $S_0$ are coefficient rings and $R_0\subset S_0$. Since $mS=n$, $S=S_0[[x_1,\ldots,x_r]]$. If $R_0$ is a field then $R_0=k_m$ and $S_0=k_n$. So $S=R\otimes_{k_m}k_n$ and hence $S$ is a free $R$-module of rank $[k_n:k_m]$. So $|D(n)|=[L:K]=[k_n:k_m]=|\Gal(k_n/k_m)|$. So the natural epimorphism from $D(n)\to \Gal(k_n/k_m)$ is an isomorphism. Hence $I(n)$ is trivial. 
  
  Suppose $R_0$ is not a field. Since $R$ and $S$ are integral domains $R_0$ and $S_0$ are characteristic zero dvrs with residue fields $K_m$ and $k_n$ are of characteristic $p>0$. Also by Cohen structure theorem, $pR_0$ and $pS_0$ are maximal ideals of $R_0$ and $S_0$ respectively. Hence $S_0$ is finite \'etale extension of $R_0$. Since $R_0$ is a dvr, $S_0$ is free $R_0$-module. Hence $S$ is a free $R$-module of rank $[L:K]$. This similarly implies that $I(n)$ is the trivial group. 
\end{proof}

\begin{df}
  Let $R$ be an excellent normal domain with fraction field $K$ and $L/K$ be a finite separable extension. Let $S$ be the integral closure of $R$ in $L$, prime ideal $Q$ of $S$ and $P=Q\cap R$. 
  The prime ideal $Q$ is said to be \emph{unramified} in the extension $S/R$ or $L/K$ if $PS_Q=QS_Q$ and $S_Q/QS_Q$ is a finite separable extension of $R_P/PR_P$. In this scenario we will also say $S_Q/R_P$ is an unramified extension. If every $Q$ lying over $P$ is unramified then $P$ (or $R_P$ ) is said to be non-branched in $L$ otherwise $P$ (or $R_P$) is said to be a \emph{branched} in $L$. 
\end{df}

\begin{rmk}
 By flatness of localization and completion $S_Q/R_P$ is unramified is equivalent to $\widehat S_Q/\widehat R_P$ is unramified. In the complete local ring set-up if the ring is understood from the context we may simply say $\QF(\widehat S_Q)/\QF(\widehat R_P)$ is unramified to mean $\widehat R_P$ is unramified in the field extension $\QF(\widehat S_Q)/\QF(\widehat R_P)$. By the Proposition \ref{unramified-closed-points} when $L/K$ is a Galois extension then $I(Q)$ is trivial iff then $Q$ is unramified in $L/K$. 
\end{rmk}

The following is a generalization of \cite[Lemma 3.1]{killwildram} to higher dimension.

\begin{lemma} \label{compositum-wild-ramification}
 Let $R$ be a normal local excellent domain and $K$ be the quotient field of $R$. Let $L$ and $M$ be finite separable extensions of $K$ and $\Omega=LM$ their compositum. Let $A_s$ be the integral closure $R$ in $\Omega$ and $A$ be the localization of $A$ at a maximal ideal of $A_s$. Then $S=A\cap L$ and $T=A\cap M$ are normal local excellent domains. Let $\widehat K$, $\widehat L$, $\widehat M$ and $\widehat \Omega$ be the quotient field of the complete local rings $\widehat R$, $\widehat S$, $\widehat T$ and $\widehat A$ respectively. Here all fields are viewed as subfields of an algebraic closure of $\widehat K$. The following holds: 
 \begin{enumerate}
  \item If $A/m_A$ is compositum of $S/m_S$ and $T/m_T$ then $\widehat \Omega = \widehat L \widehat M$.
  \item \label{eq2:compositum-wild-ramification} If $A/m_A$ is a separable extension of $S/m_S$ then $\widehat \Omega$ is an unramified extension of $\widehat L \widehat M$.
 \end{enumerate}
\end{lemma}

\begin{proof}
 Note that $A_s$ is a semilocal ring and a finite $R$-module (since excellent rings are Nagata). Let $m_R$ be the maximal ideal of $R$. Note that $\widehat A$ is a homomorphic image of $\widehat A_s$, the $m_R$-adic completion of $A_s$. So $\widehat A_s$ and hence $\widehat A$ are finite $\widehat R$-modules. 
 
 Note that $S$ and $T$ are also a localization at maximal ideals of the integral closures of $R$ in $L$ and $M$ (respectively) dominated by $A$. Hence $\widehat S$ and and $\widehat T$ are finite $\widehat R$-modules. By definition $A$, $S$, $T$ are normal local domains and they are excellent because they are localization of a finite type ring over an excellent ring.
 
 Note that $\widehat L$ and $\widehat M$ are contained in $\widehat \Omega$. So $\widehat L\widehat M \subset \widehat \Omega$. If $A/m_A$ is the compositum of $S/m_S$ and $T/m_T$ then there exist $\alpha_1,\ldots,\alpha_l\in T$ such that $A/m_A=(S/m_S)[\bar\alpha_1,\ldots,\bar\alpha_l]$. Let $\widehat S_1$ be the integral closure of $\widehat S[\alpha_1,\ldots,\alpha_l]$ in the fraction field. Then $\widehat L \subset\QF(\widehat S_1) \subset \widehat LM$. Also $\widehat S_1/m_{\widehat S_1}=A/m_A$ and $\widehat S_1$ is a normal complete local domain between $\widehat S$ and $\widehat A$. So by replacing $\widehat S$ by $\widehat S_1$, we may assume $A/m_A=S/m_S$.
 
 Let $m_A=(a_1,\ldots,a_n)$ for $a_i\in A\subset \widehat A$. Then $\widehat S[a_1,\ldots,a_n]$ is a finite $\widehat S$-module because $\widehat A$ is finite over $\widehat R$ and $\widehat R\subset \widehat S$. Hence $\widehat S[a_1,\ldots, a_n]$ is a complete local ring by \cite[Theorem 7]{Coh} with maximal ideal $m$ (say). Also $\widehat S\subset \widehat S[a_1,\ldots, a_n]\subset \widehat A$ are extensions of local rings with $\widehat S/m_{\widehat S}=\widehat A/m_{\widehat A}$. Hence the $\widehat S[a_1,\ldots,a_n]$ and $\widehat A$ have the same residue field. Also $m_{\widehat A}=m_{\widehat S}\widehat A$ since $m_{\widehat A}\supset m \supset m_A$ and $m_A\widehat A=m_{\widehat A}$. Hence by \cite[Corollary to Theorem 8]{Coh}, $\widehat S[a_1,\ldots, a_n]=\widehat A$. So $\widehat \Omega=\widehat L[a_1\ldots,a_n]\subset \widehat LM\subset \widehat L\widehat M$. This completes the proof of the first part.
 
 Let $\widehat B$ be the integral closure of $\widehat S$ in $\widehat L\widehat M$. Since $\widehat \Omega\supset \widehat L\widehat M$, $\widehat A\supset\widehat B$ and $\widehat A$ is a finite $\widehat B$-module (as $\widehat B$ contains $\widehat S$). So $m_{\widehat B}=m_{\widehat A}\cap \widehat B$. Note that $a_1,\ldots a_n \in m_{\widehat B}$ since $a_1,\ldots,a_n\in m_A\subset LM\subset \widehat L \widehat M$. So $m_{\widehat B}\widehat A=m_{\widehat A}$. Moreover $\widehat S/m_{\widehat S} \subset \widehat B/m_{\widehat B} \subset \widehat A /m_{\widehat A}$ and the hypothesis for \eqref{eq2:compositum-wild-ramification} implies $\widehat A/m_{\widehat A}$ is a separable extension of $\widehat B/m_{\widehat B}$. Hence by Proposition \ref{unramified-closed-points} $\widehat A$ is an unramified extension of $\widehat B$, i.e. $\widehat \Omega$ is an unramified extension of $\widehat L\widehat M$. 
 \end{proof}

\begin{cor}
 Let the notation be as in the above lemma. If $\widehat L\subset \widehat M$ then $A/T$ is an unramified extension. 
\end{cor}

\begin{proof}
 Since $\Omega/M$ is a finite extension, so is $\widehat \Omega/\widehat M$. Hence $\widehat A$ is a finite $\widehat T$-module. The above lemma and the hypothesis implies that $\widehat \Omega=\widehat M$. So $\widehat A=\widehat T$, i.e. $A/T$ is unramified.
\end{proof}

\section{Branch data}\label{sec:branch-data}

For an excellent normal scheme $X$ and $x\in X$, let $\cK_{X,x}$ be the fraction field of $\widehat \cO_{X,x}$. Let $\bar x$ denote the closure of $\{x\}$ in $X$. Let $\widehat X_{\bar x}$ be the formal scheme obtained by the completion of $X$ along $\bar x$. For an open affine connected subset $U$ of $X$ containing $x$, let $\cK_X^x(U)=\QF(\widehat R^I)$ where the coordinate ring of $U$ is $R$ and $I$ is the (prime) ideal of $R$ defining $x$. Note that $\cK_X^x(U)=\QF(\cO_{\widehat X_{\bar x}}(U\cap \bar x))$.

\begin{pro}\label{open-subsets-okay}
  In the above setup, for $x\in X$ and $U\subset V$ be two affine open connected neighbourhood of $x$ in $X$, $\cO_{\widehat X_{\bar x}}(U\cap \bar x) \supset \cO_{\widehat X_{\bar x}}(V\cap \bar x)$ and the completion of $\cO_{\widehat X_{\bar x},x}$ at the maximal ideal is $\widehat \cO_{X,x}$.
  In particular we have the inclusion of fields $\cK_X^x(V)\subset \cK_X^x(U) \subset \cK_{X,x}$.
\end{pro}

\begin{proof}
 The first statement follows from (\cite[Proposition 10.1.4]{ega1}) and noting that $U$ and $V$ are affine integral schemes. Since $X$ is normal excellent schemes, all the three rings are integral domains (\cite[7.8.3]{ega4.2}) and hence admit fraction fields.
\end{proof}

\begin{pro}\label{local-ramified}
  In the above setup let $L/\cK_X^x(U)$ be a finite Galois extension and $R=\cO_{\widehat X_{\bar x}}(U\cap \bar x)$. Let $S$ be the integral closure of $R$ in $L$. Let $y\in \Spec(S)$ be a point lying above $x\in \Spec(R)$ and $I$ and $J$ be the ideals defining $x\in \Spec(R)$ and $y\in \Spec(S)$ respectively. The morphism $\Spec(S)\to \Spec(R)$ is unramified at $x$ iff $\widehat S_J/\widehat R_I$ is an unramified extension of complete local rings. In particular under the assumption that $L/\cK_X^x(U)$ is Galois, $x$ is not branched in the morphism $\Spec(S)\to \Spec(R)$ iff $\widehat \cO_{X,x}$ is unbranched in the extension $L\cK_{X,x}/\cK_{X,x}$.     
\end{pro}

\begin{proof}
  By definition the morphism $\Spec(S)\to \Spec(R)$ is unramified at $x$ iff $IS_J=JS_J$ and $S_J/JS_J$ is a separable extension of $R_I/IR_I$. But passing to the completion does not change the residue fields and the identity $I\widehat S_J=J\widehat S_J$ also holds. Hence $\widehat S_J/\widehat R_I$ is unramified extension. If $L/\cK_X^x(U)$ is Galois then $x$ is not branched in the morphism $\Spec(S)\to \Spec(R)$ iff $\widehat S_J/\widehat R_I$ is an unramified extension (as all the points lying above $x$ are conjugates). But $\widehat R_I=\widehat \cO_{X,x}$ and $\QF(\widehat S_J)=L\cK_{X,x}$.  
\end{proof}

\begin{df}\label{def:branch-data}
  A \emph{quasi-branch data} on an excellent normal scheme $X$ is a function $P$ which to every point $x\in X$ of codimension at least one and an open affine connected $U\subset X$ containing $x$, assigns a finite Galois extension $P(x,U)/\cK_X^x(U)$ in some fixed algebraic closure of $\cK_X^x(U)$ such that $P(x_1,U)=P(x_2,U)\cK_X^{x_1}(U)$ whenever $x_1\in \overline{\{x_2\}}$ and for $x\in V\subset U\subset X$ with $V$ affine open connected $P(x,V)=P(x,U)\cK_X^x(V)$. 
  
  We let $P(x)=P(x,U)\cK_{X,x}$. The \emph{branch locus of $P$}, $\BL(P)=\{x\in X:\widehat \cO_{X,x} \text{ is branched in } P(x) \}$. The function $P$ will be called a \emph{branch data} if $\BL(P)$ is a closed subset of $X$ of codimension at least one. The branch data in which all the finite extensions are trivial is called the \emph{trivial branch data} and is denoted by $O$. A branch data with empty branch locus is called an {essentially trivial branch data}.
  
  Note that when $x\in X$ is a closed point then $\cK_X^x(U)$ is independent of the of the choice of $U$ and hence this definition of branch data agrees with \cite[Definition 2.1]{orbifold.bundles} when $X$ is of dimension one.
\end{df}

\begin{cor}\label{specialization-okay}
  Let $X$ be a normal integral excellent scheme with function field $K$ and $L/K$ be a finite separable extension. Let $Y$ be the normalization of $X$ in $L$ and $f:Y\to X$ be the corresponding morphism. Let $y_1, y_2$ be points of $Y$ such that $y_1\in\overline{\{y_2\}}$ and let $x_i=f(y_i)$ for $i=1,2$. Let $U$ be an open affine connected neighbourhood of $x$ and $V=f^{-1}(U)$. Then $\cK_X^{x_2}(U)\subset \cK_X^{x_1}(U)$ and $\cK_Y^{y_1}(V)=L\cK_X^{x_1}(U)=\cK_Y^{y_2}(V)K_X^{x_1}(U)$.
\end{cor}

\begin{proof}
  This is a translation of Proposition \ref{decomposition-variation}.
\end{proof}

Some of the notions and properties of branch data (\cite[2.3,2.4,2.5]{orbifold.bundles}) which hold for curves extend to higher dimension as shown below.  

\begin{pro} \label{B_f}
  Let $X$ be a normal integral excellent scheme with function field $K$ and $L/K$ be a finite Galois extension. Let $Y$ be the normalization of $X$ in $L$ and $f:Y\to X$ be the corresponding morphism. For $x\in X$ of codimension at least one and $U$ an open affine connected neighbourhood of $x$, let $y\in Y$ be such that $f(y)=x$ and $V=f^{-1}(U)$, define $B_f(x,U)=\cK_Y^y(V)$. Then $B_f$ is a branch data on $X$. Moreover if $X$ is nonsingular then $\BL(B_f)$ is either empty or pure of co-dimension one. 
\end{pro}

\begin{proof}
  Since $L/K$ is a Galois extension, all points in $f^{-1}(x)$ for any point $x\in X$ are conjugates and $\cK_Y^y(V)/\cK_X^x(U)$ are isomorphic Galois extensions for every $y\in f^{-1}(x)$. Hence $B_f(x,U)/\cK_X^x(U)$ is a Galois extension and it is independent of the choice of $y\in Y$ lying above $x$. That $B_f$ is a quasi branch data follows from Corollary \ref{specialization-okay} and Proposition \ref{open-subsets-okay}.
  
  The ramification locus of $f$ is a proper closed subset of $Y$ and $f$ is a proper morphism. Hence the branch locus of $f$ in $X$ is a proper closed subset. Hence $B_f$ is a branch data. The moreover part follows from Zariski's purity of branch locus.
\end{proof}

\begin{df}
 Let $P$ and $Q$ be quasi-branch data on $X$. Define their intersection $(P\cap Q)(x,U)=P(x,U)\cap Q(x,U)$ and their compositum $(PQ)(x,U)=P(x,U)Q(x,U)$ for all $x\in X$ of codimension at least one and $U$ any affine open neighbourhood of $x$.
\end{df}
  
\begin{lemma}\label{compositum-branch-data}
Compositum of two (quasi-)branch data is a (quasi-)branch data. 
\end{lemma}

\begin{proof}
Note that the compositum of two Galois extensions are Galois. So it follows that if $P$ and $Q$ are quasi-branch data then so is $PQ$.

Moreover $\BL(PQ)=\BL(P)\cup\BL(Q)$. So if $P$ and $Q$ are branch data then so is $PQ$.
\end{proof}


\begin{pro}\label{pull-back-branch-data}
  Let $P$ be a branch data on $X$ and $f:Y\to X$ be a quasi-finite dominant separable morphism between normal excellent schemes. For $y\in Y$, let $U$ be an affine open connected neighbourhood of $f(y)$ so that $P(f(y),U)$ is a finite Galois extension of $\cK_X^{f(y)}(U)$. Let $V$ be an affine open connected neighbourhood of $y$ contained in $f^{-1}(U)$ and $Q(y,V)=\cK_Y^y(V)P(f(y),U)$. Then $Q$ is a branch data on $Y$ and will be denoted by $f^*P$.
\end{pro}

\begin{proof}
  Let $y\in Y$ and $x=f(y)$ and $U$ be an affine open connected neighbourhood of $x$ for which $P(x,U)/\cK_X^x(U)$ is a Galois extension. Since $\cK_X^x(U)\subset \cK_Y^y(V)$ for any $V$ affine open subset of $f^{-1}(U)$ containing $y$, $Q(y,V)/\cK_Y^y(V)$ is a Galois extension. Also if $y_1\in \overline{ \{ y_2\}}$ then $f(y_1)\in \overline{ \{f(y_2)\}}$. Hence for some affine open connected neighbourhood $U$ of $f(y_1)$, $P(f(y_1),U)=P(f(y_2),U)\cK_X^{f(y_1)}(U)$. So for $V$ an affine open connected subset of $f^{-1}(U)$ containing $y_1$, $Q(y_1,V)=\cK_Y^{y_1}(V)P(f(y_1),U) = \cK_Y^{y_1}(V)P(f(y_2),U)\cK_X^{f(y_1)}(U) = \cK_Y^{y_1}(V)P(f(y_2),U)$. 
  
  By Proposition \ref{decomposition-variation} $\cK_Y^{y_2}(V)\subset \cK_Y^{y_1}(V)$ so $Q(y_1,V)=\cK_Y^{y_1}(V)P(f(y_2),U)\cK_Y^{y_2}(V)$. Hence $Q(y_1,V)=Q(y_2,V)\cK_Y^{y_1}(V)$.
 
  Finally by Remark \ref{inertia-variation} if $y\in\BL(Q)$ then every point in $\overline{ \{y\}}$ is in $\BL(Q)$. Hence $\BL(Q)$ is a closed subset of $f^{-1}(\BL(P))$. Hence $Q$ is a branch data.
\end{proof}

\begin{df}
Let $P$ and $Q$ be two branch data on a normal excellent scheme $X$. We say $P$ is less than or equal to $Q$ (and write $P\le Q$) if for all points $x\in X$ of codimension at least one and an affine open connected neighbourhood $U$ of $x$, $P(x,U)\subset Q(x,U)$. In particular $P(x)\subset Q(x)$ or all points $x\in X$ of codimension at least one. 
\end{df}

\begin{notation}
  Let $x\in X$ be a point and $\cK_{X,x}\subset L_1 \subset L_2$ be finite field extension. Then we say the extension $L_2/L_1$ is ramified (or unramified) if integral closure of $\widehat \cO_{X,x}$ in $L_2$ is ramified (or unramified) extension of integral closure of $\widehat \cO_{X,x}$ in $L_1$.  
\end{notation}

\section{Formal orbifolds} \label{sec:formal-orbifolds}

Now we are ready to extend the definition of formal orbifolds for curves and morphisms between them in \cite[Definition 2.1(3), 2.6]{orbifold.bundles} to arbitrary dimension and many of the basic properties also extend. Another difference from \cite{orbifold.bundles} in the treatment is that we allow morphisms to be quasi finite (as opposed to finite morphisms).  

\begin{df}\label{df:main}
 A \emph{formal orbifold} is a pair $(X,P)$ where $X$ is a finite type reduced normal scheme over a perfect field $k$ or $\spec(\ZZ)$ and $P$ is a branch data on $X$. 
 \begin{enumerate}
  \item An admissible morphism of formal orbifolds $f:(Y,Q)\to (X,P)$ is a quasi-finite dominant separable morphism $f:Y\to X$ of normal excellent schemes such that $Q(y,f^{-1}(U))\supset P(f(y),U)$ for all points $y\in Y$ of codimesion at least one and for some $U$ affine open connected neighbourhood of $f(y)$ with $f^{-1}(U)$ affine. \\
  \item An admissible morphism $f$ is called flat morphism if $f:Y\to X$ is flat away from $f^{-1}(\BL(P))$.\\
  \item An admissible morphism $f$ is said to be unramified or \'etale at $y$ if the extension $Q(y)/P(f(y))$ is unramified and $f$ is called \'etale morphism if $f$ is unramified for all $y\in Y$.
  \item An admissible morphism is called a covering morphism if it is also proper (and hence finite).
  \item A covering morphism $f$ is called an \'etale cover if $f$ is an \'etale covering morphism.
\end{enumerate}
\end{df}

 \begin{pro}\label{composition-etale}
  Composition of admissible (respectively \'etale) morphisms is admissible (respectively \'etale).
 \end{pro}

 \begin{proof}
   Let $f:(Y,Q)\to (X,P)$ and $g:(Z,B)\to (Y,Q)$ be admissible morphism of formal orbifolds. Then $f\circ g$ is a quasi-finite dominant separable morphism as these properties are preserved under composition. Since $f$ and $g$ are admissible, for any $z\in Z$ there exists $U$ affine open connected neighbourhood of $g(z)$ such that $g^{-1}(U)$ is affine and $B(z,g^{-1}(U))\supset Q(g(z),U)$. Similarly there exists $V$ affine open connected neighbourhood of $f\circ g(z)$ such that $f^{-1}(V)$ is affine and  $Q(g(z),f^{-1}(V))\supset P(f\circ g (z),V)$. Shrinking $V$ if necessary we may assume $f^{-1}(V)$ is an affine open subset of $U$ and $W=g^{-1}(f^{-1}(V)$ is affine. Now $B(z,W)=B(z,g^{-1}(U))\cK_Z^z(W)\supset Q(g(z),U)\cK_Y^{g(z)}(f^{-1}(V))=Q(g(z),f^{-1}(V))$ and hence $B(z,W)\supset P(f\circ g (z),V)$. Hence $f\circ g$ is an admissible morphism. 
   
   Note that being unramified is preserved under composition, hence if $f$ and $g$ are unramified at all points then so is $f\circ g$. 
 \end{proof}
 
 The following is a known result but perhaps not written down explicitly. 
 \begin{pro}
  Let $f:Y\to X$ be a quasi-finite dominant morphism between normal schemes which is unramified at $y\in Y$. Then $f$ is flat at $y$. 
\end{pro}

\begin{proof}
 Let $y\in Y$ and $x=f(y)$. Let $A$ be the strict henselianization of $\cO_{X,x}$ then $A$ is a normal domain faithfully flat over $\cO_{X,x}$. Now $C=A\otimes_{\cO_{X,x}} \cO_{Y,y}$ is a finite unramified extension of $A$ and $A$ is a strictly henselian integral domain, hence $C$ is a free $A$-module and hence flat. Since $A/\cO_{X,x}$ is faithfully flat, the original extension $\cO_{Y,y}/\cO_{X,x}$ is flat.  
\end{proof}
 
 \begin{cor}
   Let $f:(Y,Q)\to (X,P)$ be an \'etale morphism. Then $f$ is flat away from $f^{-1}(\BL(P))$.
 \end{cor}

 \begin{proof}
   Let $y\in Y$ be such that $y\notin f^{-1}(\BL(P))$ and let $x=f(y)$. Then $P(x)/\cK_{X,x}$ is unramified. Since $f$ is \'etale $Q(y)/P(x)$ is unramified. Hence $Q(y)/\cK_{X,x}$ is unramified which implies $\cK_{Y,y}/\cK_{X,x}$ is unramified. Hence $f:Y\to X$ is unramified at $y$ and $X$ is normal. This implies $f$ is flat at $y$ by the above proposition.
 \end{proof}

 \begin{pro}
  Let $f:(Y,Q)\to (X,P)$ and $g:(Z,B)\to (Y,Q)$ be flat morphisms of formal orbifolds such that $Q/f^*P$ or equivalently $id: (Y,Q)\to (Y,f^*P)$ is unramified. Then $f\circ g$ is flat.
 \end{pro}

 \begin{proof}
  Note that $\BL(Q)=\BL(f^*P)$ and by Proposition \ref{pull-back-branch-data} $\BL(f^*P)\subset f^{-1}(\BL(P))$. Hence $(f\circ g)^{-1}(\BL(P))=g^{-1}(f^{-1}(\BL(P)))\supset g^{-1}(\BL(Q))$. So if $z\notin (f\circ g)^{-1}(\BL(P))$ then $z\notin g^{-1}(\BL(Q))$, hence $g$ is flat (as morphism of schemes) at $z$ and $g(z)\notin f^{-1}(\BL(P))$. Hence $f$ is flat at $g(z)$ which implies $f\circ g$ is flat at $z$. 
 \end{proof}

\begin{df}
 Let $f:(Y,Q)\to (X,P)$ be a covering morphism of formal orbifolds. The ramification locus of $f$ is $\{y\in Y: Q(y)/P(f(y))\text{ is a ramified extension}\}$. Its image in $X$ is called the branch locus of $f$.
\end{df}

\begin{lemma}\label{branch-locus-closed}
 Let $f:(Y,Q)\to (X,P)$ be a covering morphism of formal orbifolds then the ramification locus and the branch locus of $f:(Y,Q)\to (X,P)$ are closed. 
\end{lemma}

\begin{proof}
  Since $f$ is proper it is enough to show that the ramification locus is closed. Let $y\in Y$ be in the ramification locus of $f$ and $y_0\in \overline{ \{y\}}$ be a point. Let $x=f(y)$ and $x_0=f(y_0)$. Note that $x_0$ is in the closure of $x$. Let $U$ be an affine open connected neighbourhood of $x_0$ and $V=f^{-1}(U)$. Since $f$ is ramified at $y$, by Proposition \ref{local-ramified} $y$ is ramified in the field extension $Q(y,V)/P(x,U)$. Also $P(x_0,U)=P(x,U)\cK_{X}^{x_0}(U)$ and $Q(y_0,V)=Q(y,V)\cK_{Y}^{y_0}(V)$. So we have 
  \begin{align*}
  Q(y,V)P(x_0,U)&=Q(y,V)\cK_{X}^{x_0}(U)\\
                &=Q(y,V)K_Y^y(V)K_X^{x_0}(U)\\
		&=Q(y,V)K_Y^{y_0}(V) \text{ (by Proposition \ref{decomposition-variation}).}\\
		&=Q(y_0,V)
  \end{align*}

  Let $A_0$ be the normalization of $\widehat{\cO}_{X,x_0}$ in $P(x_0)$  and $B_0$ be the normalization of $A_0$ in $Q(y_0)$. Let $I$ be the ideal of $\cO_X(U)$ defining the point $x$ and $J$ be the ideal of $\cO_Y(V)$ defining $y$. Note that $J\cap \cO_X(U)=I$. Let $A$ be the normalization of $\widehat{\cO_X(U)}^I$ in $P(x,U)$  and $B$ be the normalization of $A$ in $Q(y,V)$ (which is same as the normalization of $\widehat{\cO_Y(V)}^J$). Note that $A_0$ is the completion of the stalk of $\spec(A)$ at a point lying above $x_0$ and  similarly $B_0$.
  
  
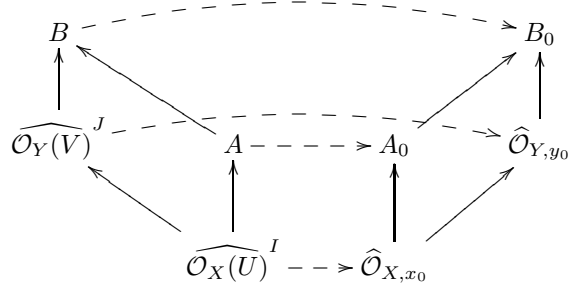
\[
    \xymatrix{
      B\ar@/^1pc/@{-->}[rrr] &                         &               & B_0\\
      \widehat{\cO_Y(V)}^J\ar[u]\ar@/^1pc/ @{-->}[rrr]    & A\ar[ul]\ar@{-->}[r] & A_0\ar[ru] & \widehat\cO_{Y,y_0}\ar[u]\\
      & \widehat{\cO_X(U)}^I\ar[lu]\ar[u]\ar@{-->}[r] & \widehat\cO_{X,x_0} \ar[u]\ar[ru]
    }
  \]
  Hence there are prime ideals in $B$ denoted by $\tilde J\subset \tilde J_0$  which corresponds to points $y$ and $y_0$ of $Y$. Moreover $\tilde J$ and $\tilde J_0$ lie over prime ideals in $A$ which corresponds to points $x$ and $x_0$ of $X$ respectively. Now using Remark \ref{inertia-variation}, the inertia group $I(\tilde J)$ in the extension $B/A$ is a subgroup of $I(\tilde J_0)$. Since $Q(y)/P(x)$ is ramified $I(\tilde J)$ is a nontrivial group and hence $I(\tilde J_0)$ is also a nontrivial group. This implies $\tilde J_0$ is ramified in $B/A$ and passing to the completion at $\tilde J_0$ one gets $Q(y_0)/P(y_0)$ is ramified. So $y_0$ is in the ramification locus of $f:(Y,Q)\to X,P)$. Hence the ramification locus is closed.
\end{proof}

\begin{lemma}
  Let $(Y,Q)$ and $(X,P)$ be formal orbifolds. Let $f:Y\to X$ be a finite morphism such that for all codimension one point $y\in Y$ and some affine open connected neighbourhood $U$ of any $x_0\in \overline{\{f(y)\}}$, $Q(y,f^{-1}(U))\supset P(f(y),U)$. Then $f:(Y,Q)\to (X,P)$ is a covering morphism of formal orbifolds.
\end{lemma}

\begin{proof}
  Let $y_0$ be a point of $Y$ of codimension at least two and $y_1$ be a codimension one point of $Y$ such that $y_0\in \overline{\{y_1\}}$. Let $x_i=f(y_i)$ for $i=0,1$ and $U$ be an affine open connected neighbourhood of $x_0$ such that $Q(y_1,V)\supset P(x_1,U)$ where $V=f^{-1}(U)$. Since $P,Q$ are branch data, $Q(y_0,V)=Q(y_1,V)\cK_Y^{y_0}(V)$ and $P(x_0,U)=P(x_1,U)\cK_X^{x_0}(U)$. By assumption $Q(y_1,V)\supset P(x_1,U)$ and since $f$ is a finite morphism $\cK_Y^{y_0}(V)\supset \cK_X^{x_0}(U)$. Hence $Q(y_0,V)\supset P(x_0,U)$.  
\end{proof}

\begin{df}
 A formal orbifold $(X,P)$ is called a geometric formal orbifold if there exists an \'etale cover $f:(Y,Q)\to (X,P)$ where $Q$ is an essentially trivial branch data on $Y$. In this case $P$ is called a geometric branch data on $X$.
\end{df}

The following analogue of \cite[Lemma 2.12]{orbifold.bundles} holds in higher dimension as well with essentially the same proof.

\begin{pro}\label{unramified-criterion}
Let $(X,P)$ be an integral formal orbifold with function field $K$ and $f:Y\to X$ be a finite covering where $Y$ is normal. Then for a branch data $Q$ on $Y$, $Q\ge f^*P$ iff $f:(Y,Q)\to (X,P)$ is a morphism of formal orbifolds. Moreover if $Q=f^*P$ then $f$ is unramified iff $\cK_{Y,y}P(f(y))/P(f(y))$ is an unramified extension for all $y\in Y$.
\end{pro}


\begin{rmk}
Let $(X,P)$ be a formal orbifold, $f:Y\to X$ be a cover and $Q=f^*P$. If $y\in \BL(Q)$ then $Q(y)/\cK_{Y,y}$ is ramified. But $Q(y)=P(f(y))\cK_{Y,y}$, hence $P(f(y))/\cK_{X,f(y)}$ is ramified. Hence $f(y)\in \BL(P)$, i.e., $\BL(Q)\subset f^{-1}(\BL(P))$
\end{rmk}

\begin{cor}\label{B_f-geometric}
Let $f:Y\to X$ be a finite covering of normal varieties. Then $(Y,O)\to (X,B_f)$ is \'etale.
\end{cor}

\begin{proof}
  By Proposition \ref{B_f}, $(X,B_f)$ is a formal orbifold. Moreover by definition of $B_f$, $B_f(f(y),f^{-1}(U))=\cK_Y^{y}(f^{-1}(U))$ for all $y\in Y$ and $U$ an affine open connected neighbourhood of $f(y)$. Note that $f^*B_f=O$. Hence by the previous proposition $(Y,O)\to (X,B_f)$ is unramified at all $y\in Y$. Hence the morphism $(Y,O)\to (X,B_f)$ of formal orbifolds is \'etale.
\end{proof}

\begin{pro} 
 Let $(X,P)$ be a formal orbifold. There exists a closed subset $Y$ in $X$ of codimension at least two and an open cover $\{U_i\}$ of $X\setminus Y$ such that $(U_i,P|_{U_i})$ is a geometric formal orbifold.
\end{pro}

\begin{proof}
  We may assume $X$ is connected and hence integral. Let $B_1,\ldots,B_r$ be the irreducible component of $\BL(P)$ of codimension one in $X$. Let $x_1,\ldots,x_r$ be the generic points of $B_1,\ldots,B_r$. Let $L_i/k(X)$ be a finite separable field extension such that $L_i\cK_{X,x_i}=P(x_i)$. Let $f_i:V_i\to X$ be the normalization of $X$ in $L_i$ and $Q_i=f_i^*P$. By construction $f_i:(V_i,Q_i)\to (X,P)$ is unramified above $x_i$ and $\BL(Q_i)$ does not contain $x_i$. Let $Z_i$ be the branch locus of $f_i:(V_i,Q_i)\to(X,P)$. Let $U_i=X\setminus (f_i(\BL(Q_i))\cup Z_i)$. Then $x_i\in U_i$. Let $U_0=X\setminus \BL(P)$. Note that $(U_i,P|_{U_i})$ is a geometric formal orbifold for $i=0,1,\ldots,r$ (for $i=0$, $id_{U_0}$ is \'etale and the branch data $P|_{U_0}$ is essentially trivial; for other $U_i$'s $f_i|_{f_i^{-1}(U_i)}$ are \'etale coverings of formal orbifolds and $Q_i|_{f_i^{-1}(U_i)}$ are essentially trivial branch data). Note that $\cup_{i=0}^r U_i$ is the complement of a codimension 2 subset of $X$.
\end{proof}

\begin{rmk}
 A formal orbifold which has an open cover by geometric formal orbifolds can possibly be interpreted as Deligne-Mumford stacks. But viewing them as varieties with branch data is more natural and useful to study the ramification theoretic properties of the varieties. This point of view is also more concrete and elementary.  
\end{rmk}

\begin{lemma}\label{fiber-product-branch-data}
Let $f:(Y,P)\to (X,B)$ and $g:(Z,Q)\to (X,B)$ be admissible morphisms between formal orbifolds. Let $W$ be the normalization of $Y\times_X Z$ and $p_Y:W\to Y$ and $p_Z:W\to Z$ be the projection morphisms and let $A=p_Y^*Pp_Z^*Q$. Then $A$ is a branch data on $W$. For a point $w\in W$, let $y=p_Y(w)$, $z=p_Z(w)$ and $x=f(y)=g(z)$. If $w$ is a closed point then $A(w)$ is an unramified extension of $P(y)Q(z)$.
\end{lemma}

\begin{proof}
  Since $p_Y$ and $p_Z$ are quasi-finite dominant separable morphisms $p^*_YP$ and $p^*_ZQ$ are branch data on $W$ by Proposition \ref{pull-back-branch-data}. Hence $A$ being compositum of two branch data is a branch data (Lemma \ref{compositum-branch-data}). For $w\in W$ a closed point, $p_Y^*P(w)=P(y)K_{W,w}$ and $p_Z^*Q(w)=Q(z)K_{W,w}$. Hence $A(w)=P(y)Q(z)K_{W,w}$. As $w$ is a closed point of $W$ and $W$ is a finite type scheme over a perfect field or $\spec(\ZZ)$, the residue field $k(w)$ is a separable extension of $k(y)$. So by Lemma \ref{compositum-wild-ramification} $\cK_{W,w}/\cK_{Y,y}\cK_{Z,z}$ is an unramified extension. Hence by base change $A(w)/P(y)Q(z)$ is an unramified extension.
\end{proof}

\begin{pro}\label{fiber-product}
 In the above setup, $p_Y:(W,A)\to (Y,P)$, $p_Z:(W,A)\to (Z,Q)$ are admissible morphisms and $(W,A)$ is the fiber product in the category of formal orbifolds where morphisms are admissible morphisms.
\end{pro}

\begin{proof}
  Since $f:Y\to X$ and $g:Z\to X$ are quasi finite dominant morphisms their base change $p_Y$ and $p_Z$ are also quasi finite dominant morphism. Also by definition of $A$, $A(w,p_Y^{-1}(U)) \supset P(p_Y(w),U)$ for some affine open connected neighbourhood $U$ of $p_Y(w)$. Hence $p_Y$ and $p_Z$ are morphisms of formal orbifold.
 
 Note that the normalized fiber product has the universal property of the fiber product in the category of normal schemes. This together with the definition of $A$ implies that $(W,A)$ is the fiber product of $f$ and $g$.
\end{proof}

\begin{pro}\label{base-change-unramified}\label{base-change-etale}
 Let $f:(Y,P)\to (X,B)$ be an \'etale morphism and $g:(Z,Q)\to (X,B)$ is an admissible morphism. Then pullback of $g$, $p_Z:(W,A)\to (Z,Q)$ is \'etale where $(W,A)$ is the formal orbifold fiber product of $f$ and $g$.
\end{pro}

\begin{proof}
  By Lemma \ref{branch-locus-closed} to check whether $p_Z$ is \'etale it is enough to check that $p_Z$ is unramified at the closed points of $W$. Let $w\in W$ be a closed point, $z=p_Z(w)$, $y=p_Y(w)$ and $x=f(y)=g(z)$. By hypothesis $P(y)/B(x)$ is an unramified extension so $P(y)Q(z)/Q(z)$ is also unramified. Since $w$ is a closed point of $W$ by Lemma \ref{fiber-product-branch-data} $A(w)/P(y)Q(z)$ is unramified. Hence $A(w)/Q(z)$ is unramified.
\end{proof}

%

\begin{cor} \label{fiber-product-etale}
 The fiber product of two \'etale covering morphisms is \'etale.
\end{cor}

\begin{proof}
 This follows from Proposition \ref{base-change-etale} and Proposition \ref{composition-etale}.
\end{proof}

\section{\'Etale fundamental group}\label{sec:etale-pi_1}
Once we have the category of \'etale coverings of a formal orbifold. The usual technique introduced by Grothendieck in \cite{sga1} (and also used in \cite[Section 2.2]{orbifold.bundles}) can be used to define the \'etale fundamental group of a connected formal orbifold. The results in this section are also generalization of some of the  results in \cite[Section 2.2]{orbifold.bundles} on fundamental group of formal orbifold curves over algebraically closed fields to higher dimension formal orbifolds. The proofs are also along the same lines using higher dimensional analogue of the required results proved in the previous sections and some minor modifications.

\begin{df}
Let $\scrX=(X,P)$ be a connected formal orbifold and let $Cov(\scrX)$ be the category of all finite \'etale covers of $\scrX$. 
Let $\tilde \scrX$ be the inverse limit of the inverse system $\{(Y_i,Q_i)\in Cov(\scrX)\}_{i\in I}$ of connected finite \'etale covers of $\scrX$. Given a geometric point $x$ of $X$ such that its image in $X$ is not in $\supp(P)$ and $\tilde x\in \tilde \scrX$ lying above $x$. The fiber functor $F_x$ from the category of finite \'etale covers $Cov(\scrX)$ to the category of sets is pro-representable by $\tilde\scrX$ via $\tilde x$. So the group of automorphism of $F_x$ is isomorphic to $\Aut(\tilde \scrX/X)$. 
Define the \'etale fundamental group of $\scrX$ to be $\Aut(\tilde \scrX/X)$ which is same as $\displaystyle{ \varprojlim_{i\in I} \Aut(Y_i/X)}$ over all connected \'etale covers $(Y_i,Q_i)\to \scrX$. This group will be denoted by $\pi_1(\scrX,x)$ or simply $\pi_1(\scrX)$.
\end{df}

\begin{pro}\label{pi-is-functorial}
Let $f:\scrY\to \scrX$ be a covering morphism of formal orbifolds. Then $f$ induces a homomorphism of their fundamental group $\pi_1(f):\pi_1(\scrY)\to \pi_1(\scrX)$. This makes $\pi_1$ a functor from connected pointed formal orbifolds to groups. 
Moreover if $f:\scrY\to\scrX$ is \'etale then $\pi_1(f)$ is injective. 
\end{pro}

\begin{proof}
The proof is same as that of \cite[Proposition 2.26]{orbifold.bundles}.
\end{proof}

\begin{rmk}\label{Rem:pi_1.of.identity}
 Let $P_1\ge P_2$ be branch data on connected $X$. Then $id:(X,P_1)\to (X,P_2)$ is a covering morphism of orbifolds. The induced map on their fundamental groups $\pi_1(id):\pi_1(X,P_1)\to \pi_1(X,P_2)$ is an epimorphism. 
\end{rmk}

\begin{thm}\label{approximation}
 Let $X$ be a normal connected variety over a perfect field or $\spec(\ZZ)$ and $U$ an open subset of $X$. Then $\pi_1(U)\cong \varprojlim\pi_1(X,P)$ where $P$ runs over all the branch data on $X$ such that $\BL(P)\subset X\setminus U$. 
\end{thm}

\begin{proof}
 Let $D=X\setminus U$. Since compositum $P_0$ of two branch data $P_1$, $P_2$ is a branch data on $X$ with $P_0\ge P_i$, $i=1,2$ and $\BL(P_1P_2)=\BL(P_1)\cup \BL(P_2)$. The collection of branch data with $\BL(P)\subset D$ is a projective system. Also if $P_1\ge P_2$ are two branch data on $X$ then the $id_X:(X,P_2)\to (X,P_1)$ is a covering morphism. By functoriality (Proposition \ref{pi-is-functorial}) it induces a group homomorphism $\pi_1(X,P_2)\to \pi_1(X,P_1)$. Hence $\pi_1(X,P)$ where $P$ runs over all the branch data on $X$ such that $\BL(P)\subset D$ is an inverse system of groups.
 
 Note that $\pi_1(U)=\varprojlim\Aut(L/k(U))$ where $L$ varies over all finite field extensions $L/k(U)$ (in a fixed separable closure) such that the normalization $V$ of $U$ in $L$ is \'etale over $U$. Any such extension $L/k(U)$ correspond to a finite morphism $f:Y\to X$ which extends $V\to U$. Note that $f: (Y,O)\to (X,B_f)$ is an \'etale covering morphism by Corollary \ref{B_f-geometric} and $\BL(B_f)\subset D$ as it is the branch locus of $f$ and $f$ is \'etale over $U$. Conversely given an \'etale covering morphism $f: (Y,Q)\to (X,P)$ with $\BL(P)\subset D$, $f$ restricted to $V:=f^{-1}(U)$ is \'etale. Hence the two inverse limits are isomorphic.
\end{proof}

\begin{pro}\label{product-geometric} 
 If $(X,P_1)$ and $(X,P_2)$ are geometric formal orbifolds then so is $(X,P_1P_2)$.
\end{pro}

\begin{proof}
  Let $f_1:(Y_1,Q_1)\to (X,P_1)$ and $f_2:(Y_2,Q_2)\to (X,P_2)$ be \'etale covering morphisms with $Q_1$ and $Q_2$ essentially trivial branch data. Let $P=P_1P_2$ and $R_i=f_i^*P$. Then $f_i:(Y_i,R_i)\to (X,P)$ are both \'etale.  Let $(W,A)$ be the fiber product of $(Y_i,R_i)\to (X,P)$ for $i=1,2$. For a point $w$ in $W$ with images $y_i$ in $Y_i$, $i=1,2$ and $x\in X$, $A(w)=R_1(y_1)R_2(y_2)\cK_{W,w}$. For $i=1,2$, since $Q_i$ is essentially trivial branch data on $Y_i$, $Q_i(y_i)/\cK_{Y,y_i}$ is unramified and hence $Q_i(y_i)\cK_{W,w}/\cK_{W,w}$ is unramified. So the compositum $Q_1(y_1)Q_2(Y_2)\cK_{W,w}$ is an unramified extension of $\cK_{W,w}$. But $R_i(y_i)\subset Q_i(y_i)$. Hence $R_1(y_1)R_2(y_2)\cK_{W,w}/\cK_{W,w}$ is unramified. Hence $A$ is essentially trivial branch data and by Corollary \ref{fiber-product-etale} $(W,A)\to (X,P)$ is \'etale.
\end{proof}

\begin{pro}\label{approximate-orbifold} 
Let $(X,P)$ be a connected formal orbifold. There exists $Q \le P$ a (unique) branch data on $X$ such that $(X,Q)$ is a geometric formal orbifold and $Q$ is maximal with these properties. The natural homomorphism $\pi_1(X,P)\to\pi_1(X,Q)$ is an isomorphism. 
\end{pro}

\begin{proof}
Let $f:(Y,R)\to (X,P)$ be a connected Galois \'etale covering morphism. Let $B_f$ be the branch data on $X$. Then $B_f\le P$. 
Let $p_1,\ldots, p_r$ be the irreducible components of $\supp(P)$. Let $Q(p_j)$ be the compositum of all $B_f(p_j)$ where $f:(Y,R)\to (X,P)$ is Galois \'etale. Then for all $1\le j\le r$, $\cK_{X,p_j} \subset Q(p_j)\subset P(p_j)$ and $Q(p_j)/\cK_{X,p_j}$ is a finite Galois extension. So there exist finitely many $f_i:(Y_i,R_i)\to (X,P)$ \'etale connected Galois covering such that $Q(p_j)=B_{f_1}(p_j)\ldots B_{f_n}(p_j)$ for $1\le j\le r$. Let $Q=B_{f_1}\ldots B_{f_n}$. Then $Q\le P$, $Q$ is a geometric branch data by Proposition \ref{product-geometric} and Corollary \ref{B_f-geometric}. By construction of $Q$ it is maximal geometric branch data bounded by $P$ and the homomorphism $\pi_1(X,P)\to \pi_1(X,Q)$ is an isomorphism.
\end{proof}

\begin{cor}
 Let $(X,P)$ be a connected formal orbifold then $\pi_1(X,P)$ is an extension of a finite group by $\pi_1(X')$ the \'etale fundamental group of a connected normal variety $X'$.
\end{cor}

\begin{proof}
 By Proposition \ref{approximate-orbifold} there exists a geometric branch data $Q$ on $X$ such that $\pi_1(X,P)=\pi_1(X,Q)$. Since $(X,Q)$ is a geometric formal orbifold there exists a finite \'etale Galois covering $f:(Y,R)\to (X,Q)$ (with Galois group $G$ say) where $R$ is an essentially trivial branch data. Then $\pi_1(Y,R)$ is a normal subgroup of $\pi_1(X,Q)$ with quotient isomorphic to $G$ (Proposition \ref{pi-is-functorial}). Since $R$ is essentially trivial branch data $\pi_1(Y,R)=\pi_1(Y)$. 
\end{proof}

Recall that a profinite group $\Pi$ is said to be of finite rank if there exists a finite subset $S$ of $\Pi$ such that the subgroup generated by $S$ is dense in $\Pi$. 

\begin{cor}\label{pi_1-finite-rank}
 Let $(X,P)$ be a connected formal orbifold with $X$ projective over $\spec(k)$ where the field $k$ is such that the absolute Galois group $G_k$ has finite rank then $\pi_1(X,P)$ has finite rank.
\end{cor}

\begin{proof}
 It follows from the above corollary and the fact that $\pi_1(X')$ is topologically finitely generated for any connected normal projective variety $X'$ over $\spec(k)$. (This is true for curves via Grothendieck lifting. For higher dimensional normal projective variety $X'$ one uses Lefschetz's theorem for fundamental groups, i.e. there exists a general projective curve $C$ in $X'$ such that $\pi_1(C)\to \pi_1(X')$ is surjective.) 
\end{proof}

\section{ \'Etale Sheaves and cohomology}\label{sec:site-sheaves}
In this section using the notion of \'etale morphism for formal orbifolds we briefly define small \'etale site on formal orbifolds, sheaves on them and their cohomology groups using the standard technique (see \cite[Chapter II and III]{Milne-etale-cohomology} or \cite[Chapter 20 and 21]{stacks.project}). We also prove some of the results in the formal orbifold setup which will be used in the later sections. Note that even if $f:(Y,Q)\to (X,P)$ is an \'etale morphism of formal orbifolds, $f:Y\to X$ viewed as a morphism of varieties need not be flat. 

Let $\scrX=(X,P)$  be a formal orbifold. An \'etale $\scrX$-formal orbifold is an \'etale morphism $f:(Y,Q)\to \scrX$ and morphisms between \'etale $\scrX$-formal orbifolds is an \'etale $X$-morphism. This category will be denoted by $E/\scrX$. An open \'etale cover of $\scrX$ is a collection of \'etale morphisms $\{f_i:(U_i,Q_i)\to \scrX, i\in I\}$ such that $\cup_{i\in I}f_i(U_i)=X$. This defines an \'etale topology on the category of \'etale $\scrX$-formal orbifolds. In other words this defines a \emph{small \'etale site} $E/\scrX$ on $\scrX$.

As usual a \emph{presheaf} on $E/\scrX$ is defined to be a contravariant functor from $E/\scrX$ to the category of abelian groups. A morphism $\phi$ of presheaves $\scrF$ and $\scrG$ is a natural transformation of functors where maps $\phi_U:\scrF(U)\to \scrG(U)$ for an object $U$ in $E/\scrX$ is a group homomorphism. Note that the category of presheaves is an abelian category. The argument is same as in the usual Grothendieck topology on schemes and just uses the fact that this category of functors shares most of the properties of the category of abelian groups.

 A \emph{sheaf} $\scrF$ is a presheaf such that for any $\scrU$ in $E/X$ and any \'etale cover $\{f_i:\scrU_i\to \scrU\}$, the following sequence is exact.
 \begin{equation}\label{eq:gluing}
  F(\scrU)\to \Pi_i F(\scrU_i)\rightrightarrows \Pi_{i,j} F(\scrU_i\times_{\scrU}\scrU_j)  
 \end{equation}

 \begin{ex}\label{constant-sheaf}
   Let $A$ be an abelian group. Consider the presheaf $\scrF_A$ on $\scrX$ which sends $\scrU\to \scrX$ to $\scrF(\scrU)=\{s:\scrU\to A| s \text{ is continuous}\}$ where we give $A$ the discrete topology and $\scrU=(U,P)$ the Zariski topology on $U$. And which sends an \'etale $\scrX$-morphism between \'etale $\scrX$-formal orbifolds $f:\scrU\to \scrV$ to the group homomorphism $\scrF_A(f):\scrF_A(\scrV)\to\scrF_A(\scrU)$ where $\scrF_A(f)(s)=s\circ f$. The gluing condition \eqref{eq:gluing} for such a presheaf is automatic hence it is a sheaf and it will be denoted by $A$ as well and will be called a \emph{constant sheaf}. Note that $\scrF_A(\scrU)$ is isomorphic to $r$ copies of $A$ where $r$ is the number of connected components of $\scrU$.
 \end{ex}
 
 \subsection{\v Cech cohomology}
 Let $\scrF$ be a presheaf on $\scrX$ and $\cU=\{\scrU_i\to\scrX\}$ be a covering. We define the \v Cech cohomology groups $\check H^*(\cU,\scrF)$ of $\scrF$ with respect to the covering $\cU$ in the usual way. For $n\ge0$, let $\scrU_{i_0,\ldots,i_n}=\scrU_{i_0}\times_{\scrX}\scrU_{i_1}\times_{\scrX}\ldots\times_{\scrX}U_{i_n}$ denote the $n$-fold intersection.  Consider the \v Cech complex $C^*(\cU,\scrF)$ whose $n^{\text{th}}$ term is $C^n=\Pi_{i_0<\ldots<i_n}\scrF(\scrU_{i_0,\ldots,i_n})$. 
 
 The boundary maps $d:C^n\to C^{n+1}$ are given by the alternating sum of the natural maps $\scrF(\scrU_{i_0,\ldots,\hat i_k,\ldots,i_n})\to\scrF(\scrU_{i_0,\ldots,i_n})$ induced from the natural projection morphism $\scrU_{i_0,\ldots,i_n}\to \scrU_{i_0,\ldots,\hat i_k,\ldots,i_n}$.
 
 The homology of this complex is $\check H^*(\cU,\scrF)$. If $\cV=\{\scrV_i\to \scrX\}$ is another covering of $\scrX$ and a refinement of $\cU$. Then there is a map of complexes $C^*(\cU,\scrF)\to C^*(\cV,\scrF)$ which induces maps on the cohomology groups $\check H^*(\cU,\scrF)\to \check H^*(\cV,\scrF)$. This forms a directed system of groups as the coverings of $\scrX$ vary. Define \v Cech cohomology groups of $X$ for a presheaf $\scrF$ as: 
$$\check H^*(X,\scrF)=\varinjlim_{\cU}H^*(\cU,\scrF).$$
 
 Note that for a sheaf $\scrF$, $H^0(\scrX,\scrF)=\scrF(\scrX)$. For a presheaf $\scrF$, define $\scrF^+$ to be the functor $\scrF^+(\scrU)=\check H^0(\scrU,\scrF|_{\scrU})$. There is a natural map of presheaves $\scrF\to\scrF^+$.
 \begin{thm} \cite[Theorem 7.10.10]{stacks.project}\label{sheafification}
   Let $\scrF$ be a presheaf on $\scrX$. Then
   \begin{enumerate}
    \item The presheaf $\scrF^+$ is a separated presheaf.
    \item If $\scrF$ is a separated presheaf then $\scrF^+$ is a sheaf and the natural map $\scrF\to\scrF^+$ is injective.
    \item If $\scrF$ is a sheaf then the natural map $\scrF\to\scrF^+$ is an isomorphism.
    \item The presheaf $\scrF^{++}$ is always a sheaf and $\scrF\to\scrF^{++}$ is called the \emph{sheafification} of $\scrF$.
  \end{enumerate}
 \end{thm}
 
 Let $A$ be an abelian group. The constant presheaf described in Example \ref{constant-sheaf} is a sheaf. So $\check H^0(\scrX,A)=A^r$ where $r$ is the number of connected component of $\scrX$.
 
 \begin{pro}\label{H^1-pi_1-duality}
 Let $\scrX=(X,P)$ be a formal orbifold such that $X$ is connected and projective over a field $k$ and $G_k$ is a finite rank profinite group. Then $\check H^1(\scrX,A)\cong\Hom(\pi_1(\scrX),A)$ for any finite abelian group $A$.   
 \end{pro}

 \begin{proof}
   Since $A$ is a finite abelian group and $\pi_1(\scrX)$ is finite rank (by Corollary \ref{pi_1-finite-rank}) $\Hom(\pi_1(\scrX),A)$ is a finite group. Hence there exists a finite $G$-Galois \'etale cover $\scrY\to \scrX$ such that every element of $\Hom(\pi_1(\scrX),A)$ factors through $G$. As in \cite[Chapter III, Example 2.6, page 99]{Milne-etale-cohomology} the \v Cech complex with respect to the cover $\{\scrY\to \scrX\}$ of $A$ is isomorphic to cochain complex used to compute the group cohomology $H^*(G,A)$. Hence $H^1(\scrY\to\scrX,A)\cong H^1(G,A)$. Since $A$ is a constant sheaf $G$ acts trivially on $A$ and hence $H^1(G,A)=\Hom(G,A)$. But $\Hom(G,A)$ is same as $\Hom(\pi_1(\scrX),A)$ by choice of $Y$ and $G$.
   Note that any refinement of $\scrY\to \scrX$ has a further refinement by a $G'$-Galois cover $\scrY'\to\scrX$ and hence dominates $\scrY\to\scrX$. So by the same argument $\check H^1(\scrY'\to\scrX,A)\cong Hom(\pi_1(\scrX),A)$. Hence $\check H^1(\scrX,A)\cong Hom(\pi_1(\scrX),A)$.
 \end{proof}
 
 \begin{rmk}
   This result holds without the assumption on the base field and the assumption on $X$ is projective. Note that in absence of these assumptions $\Hom(\pi_1(\scrX),A)$ will not be a finite group. But $\pi_1(\scrX)$ is a Galois group of an infinite Galois extension.  Hence one has to work with a inverse system of finite Galois covers and modify the argument in \cite[Chapter III, Example 2.6, page 99]{Milne-etale-cohomology} to get the result in this general setup.  
 \end{rmk}

 \subsection{Structure sheaf}
 
 \begin{lemma}
  Let $A\to B$ be a finite separable extension of normal domains. Then the sequence $A\to B\rightrightarrows B\times_A B$ is exact.
 \end{lemma}
 \begin{proof}
   Since $A\to B$ is generically \'etale there exists a multiplicative subset $S$ in $A$ such that $S^{-1}A\to S^{-1}B$ is a finite \'etale morphism. Hence the sequence $S^{-1}A\to S^{-1}B\rightrightarrows S^{-1}B\otimes_{S^{-1}A}S^{-1}B$ is exact.
   Note that $A$ is contained in the kernel of the map $B\to B\otimes_A B$ which sends $b\in B$ to $1\otimes b-b\otimes 1$. So it is enough to show that if $b\in B$ is such that $b\otimes 1 - 1 \otimes b=0$ in $B\otimes_A B$ then $b\in A$. But for such a $b$, the exactness after localization implies that $b\in S^{-1}A$. But $b$ is integral over $A$ and $A$ is normal implies that $b\in A$.
 \end{proof}

 Let $\scrX=(X,P)$ be a formal orbifold. Consider the structure presheaf $\cO_{\scrX}$ which sends $\scrU=(U,Q)$ to $\cO_U(U)$.
 \begin{pro}
   The structure presheaf $\cO_{\scrX}$ of a formal orbifold is a sheaf.
 \end{pro}
 
 \begin{proof}
   Let $(Y_i,Q_i)\to (X,P)$ be a covering then $\coprod Y_i\to X$ is a finite separable morphism of normal schemes. Hence the exactness of the sheaf condition for the structure presheaf follows from the above lemma and \cite[Chapter II, Proposition 1.5]{Milne-etale-cohomology}. 
 \end{proof}

%

 Once the structure sheaf of a formal orbifold $\scrX$ is defined, sheaf of $\cO_{\scrX}$-modules are simply sheaf of abelian groups $\scrF$ such that $\scrF(\scrU)$ is an $\cO_{\scrX}(\scrU)$-module for every \'etale morphism $\scrU\to \scrX$.

 \subsection{Pullback and pushforward} 
 
 The pullback and the pushforward of sheaves on formal orbifolds with respect to admissible morphisms of formal orbifolds can be defined in the same way as for the schemes in \cite[Chapter II]{Milne-etale-cohomology}. 
 
 Let $f:\scrY\to\scrX$ be an admissible morphism of formal orbifolds. Let $\scrF$ be a presheaf on $\scrY$ then $f_*\scrF$ is the presheaf given by the functor which sends $\scrU$ in $E/\scrX$ to $\scrF(\scrU\times_{\scrX}\scrY)$. This functor from the category of presheaves on $\scrX$ to the category of presheaves on $\scrY$ is called a direct image functor. Its adjoint $f^p$ is a functor from presheaves on $\scrX$ to presheaves on $\scrY$ called the inverse image functor. If $\scrF$ is a sheaf then $f_*\scrF$ is also a sheaf (cf. \cite[Chapter II, Proposition 2.7]{Milne-etale-cohomology}). The sheaf $f_*\scrF$ is also called the pushforward of $\scrF$. But $f^p$ need not send sheaves to sheaves and hence the pullback is defined as the sheafification of the inverse image functor. 
 
 \subsection{Sheaf cohomology}
 
 The \'etale site $E/\scrX$ on a formal orbifold $\scrX$ has a final object and hence the global section functor $\Gamma(\scrX,-)$ from the category of sheaves of abelian groups on $\scrX$, $\Sh(\scrX)$ to the category of abelian groups $Ab$ is given by $\scrF\mapsto \scrF(\scrX)$. By \cite[Theorem 19.7.4]{stacks.project} there are enough injectives in $\Sh(\scrX)$. Hence we can define the sheaf cohomology $H^i(\scrX,\scrF)$ as the right derived functor of the global section functor. As usual the \v Cech cohomology and the sheaf cohomology agree for degree 0 and 1 (\cite[Chapter III, Corollary 2.10]{Milne-etale-cohomology}). Hence in view of Theorem \ref{H^1-pi_1-duality} we get that $H^1(\scrX,A)\cong\Hom(\pi_1(\scrX),A)$.

 One could directly show the isomorphism $H^1(\scrX,A)\cong\Hom(\pi_1(\scrX),A)$. This can be done using Hochschild-Serre spectral sequence imitating the proof in the \'etale case as carried out in \cite[Chapter III, Theorem 2.20 and Remark 2.21(b)]{Milne-etale-cohomology}.
 
 \begin{pro}\label{spectral-sequence}
 Let $\scrX=(X,P)$ be a geometric formal orbifold and $f:(Y,O)\to (X,P)$ be a finite \'etale Galois $G$-cover. Let $\scrF$ be a sheaf on $\scrX$. Then there is a spectral sequence 
 $$H^p(G,H^q(Y,\scrF|_Y))\Rightarrow H^{p+q}(\scrX,\scrF)$$
 Here $H^q(Y,\scrF|_Y)$ is the usual \'etale cohomology.
\end{pro}
\begin{proof}
  The global section functor $\Gamma(\scrX,-)$ from $\Sh(\scrX)$ to abelian groups is the composition $(-)^G\circ\Gamma(Y,-)$ of global sections over $Y$ followed by $G$-invariants. Note that $G$ acts naturally of $\scrF|_Y$ and hence $\Gamma(Y,-)$ is a functor from $\Sh(\scrX)$ to $G$-modules. Note that $\Gamma(Y,-)$ takes injective objects in $\Sh(\scrX)$ to injective objects in the category of $G$-modules. The result now follows from Grothendieck spectral sequence.
\end{proof}

\begin{rmk}
 Since $Y$ has finite cohomological dimension. This is a bounded spectral sequence. 
\end{rmk}

\begin{cor}
Let $\scrX$ be a proper geometric formal orbifold and $\scrF$ on $X$ be such that $\scrF|_Y$ is coherent for some \'etale Galois cover $(Y,O)\to \scrX$. Then $H^i(\scrX,\scrF)$ is finite dimensional for all $i$.
\end{cor}

\begin{proof}
 Note that if $X$ is proper over $k$ then $Y$ is also proper over $k$. In this situation, $H^q(Y,\scrF|_Y)$ are finite dimensional if $\scrF|_Y$ is coherent and hence $H^i(\scrX,\scrF)$ are also finite dimensional. 
\end{proof}

\section{Brylinski-Kato filtration} \label{sec:BK-filtration}

We recall the Brylinski-Kato filtration on the ring of finite length Witt vectors over a henselian discrete valuation field $K$ of characteristic $p$ from \cite{Kato}, \cite{matsuda} and \cite{Kerz-Saito.CLT}. For $m\ge 0$,
 $$fil_m^{log}W_s(K)=\{(a_{s-1},\ldots,a_0)\in W_s(K)| p^i\ord_K(a_i)\ge -n \forall i\}$$
 Note that this is an increasing filtration with $fil_0^{log}W_s(K)=W_s(\cO_K)$. A modification of this filtration by Matsuda was used in \cite{Kerz-Saito.CLT}. Let $V:W_s(K)\to W_{s+1}(K)$ be the function sends $(a_{s-1},\ldots,a_0)$ to $(0,a_{s-1},\ldots,a_0)$ called the Verschiebung. Let $s'=\min(\ord_p(m+1),s)$ then define
 $$fil_mW_s(K)=V^{s-s'}fil_{m+1}^{log}W_{s'}(K)+fil_m^{log}W_s(K).$$
 
 Let $\delta_s:W_s(K)/(Frob-Id)\to H^1(K,\Z/p^s\Z)$ be the isomorphism given by Artin-Schreier-Witt correspondence. The filtration on the Witt rings induce a filtration on $H^1(K):=H^1(K,\Q/\Z)$ as follows:
 $$fil_m H^1(K)=H^1(K)(p')+\cup_{s\ge 1}\delta_s(fil_m W_s(K)) \text{ for }m\ge 1.$$ 
 Let $fil_0 H^1(K)$ be the subgroup of unramified characters. This filtration has a shift in numbering from \cite{matsuda} but it is consistent with \cite{Kerz-Saito.CLT} and \cite{Kerz-Saito.Lefschetz}.
 
 \begin{df}
   \begin{enumerate}
    \item For a finite abelian field extension $L/K$ of $p$-exponent less than $n$, let $W_L$ be the subgroup of $W_n(K)/Frob-Id$ corresponding to the $p$-part of the extension $L/K$ via the Artin-Schreier-Witt theory. Define the swan conductor $$Sw(L/K)=\min\{m: W_L\subset fil_m W_n(K)/(Frob-Id)\}.$$
    
    \item For a branch data $P$ on a normal variety $X$, define the swan divisor of $(X,P)$ $$Sw(X,P)=\sum_{x\in X^{(1)}}Sw(P(x)/\cK_{X,x})x,$$ where $X^{(1)}$ is the set of codimension one points in $X$.   
   \end{enumerate}
 \end{df}

  Let $\scrX=(X,P)$ be a formal orbifold over a perfect field $k$. Assume that $P(x)/\cK_{X,x}$ is an abelian extension for all $x\in X$. Let $U=X\setminus \BL(P)$.
  Set $H^1(\scrX):=\varprojlim_n H^1(\scrX,\Z/n\Z)$. Note that $\pi_1^{ab}(\scrX)\cong\Hom(H^1(\scrX),\QQ/\ZZ)$ by Pontriagan duality. Let $D$ be an effective Cartier divisor on $X$ supported on $\BL(P)$. We recall the definition of $fil_D H^1(U)$ as in \cite{Kerz-Saito.Lefschetz}.

 \begin{df}
   For $\chi \in H^1(U)$ to be in $fil_D H^1(U)$ the following must be true. For all integral curves $Z$ in $X$ not contained in $\BL(P)$ and for any closed point $x$ in the normalization $\bar Z$ of $Z$ which lies above $Z\cap \BL(P)$, $\chi|G_x \in fil_{m(x,D)} H^1(\cK_{\bar Z,x})$. Here $G_x=\Gal(\cK_{\bar Z,x}^s/\cK_{\bar Z,x})$ and $m(x,D)$ is the multiplicity of $x$ in the pullback of $D$ under the composition $\bar Z\to Z\to X$. 
 \end{df}


 \begin{df}
   Define $\pi_1^{ab}(X,D)= \Hom(fil_D H^1(U),\Q/\Z)$. This is a quotient of $\pi_1^{ab}(U)$ and corresponds to abelian covers of $X$ \'etale over $U$ whose ramification over $\BL(P)$ is ``bounded by $D$''.
 \end{df}

 \begin{thm}
  Let $\scrX=(X,P)$ be a proper connected formal orbifold such that $Z=\BL(P)$ is a simple normal crossing divisor on $X$ and let $U=X\setminus Z$.
  The image of the inclusion $H^1(\scrX)\to H^1(U)$ lies in $fil_D H^1(U)$ where the divisor $D$ on $X$ is given by 
  $D=Sw(X,P)$. Moreover for a Cartier divisor $D$ on $X$ whose support is a normal crossing divisor, $$\pi_1^{ab}(X,D)=\varprojlim \pi_1^{ab}(X,Q)$$ where $Q$ varies over all branch data such that $Q$ is abelian and $Sw(X,Q) \le D$.
 \end{thm}

 \begin{proof}
   Let $\chi$ be in the image of $H^1(\scrX)\to H^1(U)$. Let $K=k(U)$, $K^{ur,\scrX,ab}$ be the compositum of function fields of abelian \'etale covers of $\scrX$ and $K^{ur,U,ab}$ be the compositum of function fields of abelian \'etale covers of $U$. Then $\pi_1^{ab}(\scrX)=\Gal(K^{ur,\scrX,ab}/K)$ and $\pi_1^{ab}(U)=\Gal(K^{ur,U,ab}/K)$.
   
   The hypothesis on the character $\chi\in \Hom(\pi_1^{ab}(U),\Q/\Z)$ implies that it factors through $\pi_1^{ab}(\scrX)$. Let $x\in X$ be a codimension one point lying in $\BL(P)$. Note that the abelian part of the decomposition group at $x$, $G_x=\Gal(\cK_{X,x}K^{ur,U,ab}/\cK_{X,x}) \le \pi_1^{ab}(U)$. Note that the natural map $H^1(U)\to H^1(\cK_{X,x})$ is given by $\chi\mapsto \chi_x:=\chi|_{G_x}$.
   
   In view of \cite[Proposition 2.5]{Kerz-Saito.Lefschetz}, to show that $\chi\in fil_D H^1(U)$, it is enough to show $\chi_x\in fil_{m_x} H^1(\cK_{X,x})$ where $m_x=Sw(P(x)/\cK_{X,x})$. Again since $\chi$ factors through $\pi_1^{ab}(\scrX)$, $\chi_x$ factors through $\Gal(\cK_{X,x}K^{ur,\scrX,ab}/\cK_{X,x})$. But $\cK_{X,x}K^{ur,\scrX,ab}=P(x)L$ where $L/\cK_{X,x}$ is an unramified extension (over $\widehat\cO_{X,x}$). Hence $$\chi_x \in \cup_s\delta_s(W_{P(x)})+fil_0 H^1(\cK_{X,x})$$ because by definition of Brylinski-Kato filtration on $H^1(\cK_{X,x})$ the unramified characters make up $fil_0 H^1(\cK_{X,x})$. Hence $\chi_x\in fil_m H^1(U)$ if  $W_{P(x)}\subset fil_m W_s(\cK_{X,x})$ for some $s$. And this happens if $m\ge m_x=Sw(P(x)/\cK_{X,x})$.
   
   For the moreover part, let $U=X\setminus \supp(D)$. Note that given any $n\ge 1$ and any $\alpha\in fil_DH^1(U,\Z/n\Z)$ there exists a branch data $Q$ on $X$ with $Sw(X,Q)\le D$ such that $\alpha$ is the image of $H^1((X,Q),\Z/n\Z)\to H^1(U,\Z/n\Z))$. The result now follows by taking the limit and applying Pontriagan duality.  
 \end{proof}

 It leads to some natural questions for formal orbifolds.
 \begin{question}\label{Q:CFT}
  Let $\scrX$ be a geometrically connected proper geometric formal orbifold over a finite field. Does class field theory hold for for $\scrX$?
 \end{question}
 A geometric definition of chow groups (of zero cycles) for formal orbifolds is needed to make this more precise. One way to approach this problem is to realize that $\pi_1^{ab}(\scrX)$ is a quotient of $\pi_1^{ab}(X,D)$ for $D=Sw(\scrX)$. So by Kerz-Saito's class field theory chow groups of zero cycles on $\scrX$ should be an appropriate quotient of chow groups of $X$ with modulus as defined in \cite{Kerz-Saito.CLT}. And one could try to provide a geometric or hopefully a motivic interpretation of the this quotient group directly in terms of $\scrX$. 

 It will be desirable to obtain a class field theory for formal orbifolds such that the inverse limits of the cycle class maps yield the cycle class map of Kerz and Saito's class field theory. A starting point would be to show a version Lefschetz's theorem on fundamental groups for formal orbifolds. For $\pi_1^{ab}(X,D)$ it was proved by Kerz and Saito in \cite{Kerz-Saito.Lefschetz}. 
 
 \section{Grothendieck-Lefschetz theorem for fundamental groups} \label{sec:Lefschetz}
 For a formal orbifold $(X,P)$, let $\Cov(X,P)$ denote the category of finite \'etale covers $(X,P)$. 
 The question of interest in this section is the following: 
 \begin{question}\label{Q:Lefschetz}
   Let $\scrX=(X,P)$ be a geometrically connected projective smooth geometric formal orbifold of dimension $n\ge 2$ over a perfect field $k$. Does there exist a connected smooth hypersurface $Y$ on $X$ with a geometric branch data $Q$ such that the functor from $\Cov(X,P)\to \Cov(Y,Q)$, given by the normalized pullback, is fully faithful? In particular, is the induced homomorphism $\pi_1(Y,Q)\to \pi_1(X,P)$ surjective? Moreover, if dimension $n>2$ then can one find a $(Y,Q)$ such that the functor $\Cov(X,P)\to \Cov(Y,Q)$ is an equivalence? In particular, is the map $\pi_1(Y,Q)\to \pi_1(X,P)$ an isomorphism?
 \end{question}
 
 The above question can be considered as the Grothendieck-Lefschetz theorem for wildly ramified covers. In this section we make some progress towards answering this question. 
 
 Let $Y$ be a normal connected hypersurface in $X$ not contained in $\BL(P)$. Let $\scrp$ be the ideal sheaf on $X$ defining $Y$. Since $Y$ is normal and hence unibranched for any codimension one point $x\in Y$ and any affine connected neighbourhood $U$ of $x$ in $X$, the ideal $\scrp\widehat{\cO_{X}(U)}^x$ is a prime ideal (defining $Y$ in the formal neighborhood of $x$). Let $R$ be the integral closure of $\widehat{\cO_{X}(U)}^x$ in $P(x,U)$ and $\scrq_1,\ldots,\scrq_r$ be the prime ideals of $R$ lying above $\scrp$. Define $Q'(x,U\cap Y)$ for $x\in Y$ of codimension at least one to be the compositum of the Galois extensions $\QF(R/\scrq_i)$ of $\QF(\widehat{\cO_X(U)}^x/\scrp)$. Then it is easy to see that $Q'$ is a branch data on $Y$. 
 \begin{df}
   The branch data $Q'$ on $Y$ is called the restriction of the branch data $P$ and is denoted by $P_{|Y}$.
 \end{df}

 \begin{pro}\label{etale-pullback-hypersurface}
  Let $(X,P)$ be a formal orbifold and $Y$ be a normal connected hypersurface in $X$ not contained in $\BL(P)$. Let $Q$ be the maximal geometric branch data on $Y$ which is less than or equal to $P_{|Y}$ (obtained using Proposition \ref{approximate-orbifold}). Let $h:(Z,O)\to (X,P)$ be an \'etale covering. Let $W=Y\times_X Z$, $\tilde W$ the normalization of $W$ and $g:\tilde W\to Y$ the natural projection morphism then the induced morphism $g:(\tilde W,g^*Q)\to (Y,Q)$ is \'etale.
 \end{pro}

 \begin{proof}
   In view of Proposition \ref{unramified-criterion}, it is enough to show that for any point $w\in \tilde W$ of codimension at least one $\cK_{\tilde W,w}Q(g(w))/Q(g(w))$ is an unramified extension. Let $x$ and $z$ be the images of $w$ in $X$ and $Z$ respectively under the natural map. Note that $x\in Y$ and $g(w)=h(z)=x$. Note that $\cO_{\tilde W,w}$ is the normalization of $\cO_{Y,x}\otimes_{\cO_{X,x}}\cO_{Z,z}$. Since $\scrp$ is the defining ideal of $Y$ in $X$, $\cO_{Y,x}\otimes_{\cO_{X,x}}\cO_{Z,z}\cong \cO_{Z,z}/\scrp\cO_{Z,z}$. Hence $\cO_{\tilde W,w}$ is the normalization $\cO_{Z,z}/\scrp\cO_{Z,z}$. Since $Y\nsubseteq \BL(P)$ and $\scrq_1,\ldots,\scrq_r$ are the primes in $\cO_{Z,z}$ lying above $\scrp$, $\widehat \cO_{\tilde W,w}$ is the normalization of $\widehat \cO_{Z,z}/\scrq_i$ for some $i$. Hence $\cK_{\tilde W,w}=\QF(\widehat \cO_{Z,z}/\scrq_i)\subseteq P_{|Y}(x)$. Since $Q$ is the maximal geometric branch data on $Y$ among those less than or equal to $P_{|Y}$, $\cK_{\tilde W,w}\subseteq Q(x)$.
 \end{proof}


 Let $\widehat X$ be a normal excellent formal scheme. A branch data $\hat P$ on $\widehat X$ is defined in the same way as for schemes as follows. Let $x$ be a codimension one point of $\widehat X$ (i.e. codimension one in the closed fiber) and $\widehat U$ be an open affine neighbourhood of $x$ in $\widehat X$ given by $\spf(A)$ then set $\hat P(x, \widehat U)$ is a finite Galois extension of $\QF(\widehat{A}^x)$ where $\widehat A^x$ is the completion of $A$ w.r.t the prime ideal of $A$ defining the point $x$. The assignment $\hat P$ is called a quasi branch data on $\widehat X$ if $\hat P$ is compatible with affine open neighbourhoods of $x$,  $\widehat V\subset \widehat U \subset \hat X$ and with respect to specialization $x_1\subset \overline{\{x_2\}}$ of codimension at least one points of $\widehat X$ in the same way as in Definition \ref{def:branch-data}. The definition of branch locus $\BL(\hat P)$ is also the same as in Definition \ref{def:branch-data} and similarly a quasi branch data $\hat P$ on $\widehat X$ is called a branch data if $\BL(\hat P)$ is closed in the $\widehat X$.  
 
 Note that for a point $x\in \spf(A)$, $I\subset A$ be the ideal of definition of $\spf(A)$ and $I(x)$ be the ideal of definition of $x$ then $I(x)\supset I$. Let $A$ be a normal excellent ring and $J\supset I$ be ideals of $A$ then $\widehat{\widehat A^I}^J\cong \widehat A^J$. Let $(X,P)$ be a formal orbifold and $Y \subset X$ be a closed irreducible subset not contained in $\BL(P)$. Let $\widehat X_Y$ be the completion of $X$ along $Y$. The branch data $P$ induces a branch data $\widehat P$ on the formal scheme $\widehat X_Y$. More precisely for a point $y\in Y$ of codimension at least one and $U\subset X$ an affine connected neighbourhood of $y$, define $\widehat P(y, U\cap Y)=P(y,U)$. Note that $\widehat{\cO_X(U)}^y=\widehat{\cO_{\widehat X_Y}(U\cap Y)}^y$, so $\widehat P(y,U)$ is a finite Galois extension of $\cK_{\widehat X_Y}(U\cap Y)^y=\QF(\widehat{\cO_X(U)}^y)$. Note that $\BL(\widehat P)=\BL(P)\cap Y$.

\begin{df}
  We will call the pair $(\widehat X_Y,\hat P)$ as $Y$-adic completion of $(X,P)$. 
  Let $\widehat X$ be an excellent normal formal scheme and $\hat P$ be a branch data of $\hat X$. Like in \cite[3.1.6]{Grothendieck-Murre} a finite morphism of formal scheme is an adic morphism (i.e. pull back of a sheaf of ideal of definition is a sheaf of ideal of definition) of formal schemes which induces a finite morphism on the closed fibers. As mentioned in loc. cit. the category of finite coverings of $\widehat X$ is equivalent to the category of sheaf of finite $\cO_{\widehat X}$-algebras $\scrA$ on $\widehat X$. Let $f: \widehat Y\to \widehat X$ be a finite morphism of excellent normal formal schemes and let $\hat P$ and $\hat Q$ be branch data on $\widehat X$ and $\widehat Y$ respectively. Like in Definition \ref{df:main}, $f:(\widehat Y, \hat Q)\to (\widehat X,\hat P)$ is said to be a cover if for all $y\in \widehat Y$ of codimension at least one and some open affine neighbourhood $U$ of $f(y)$,  $\hat Q(y,f^{-1}(U))\supset \hat P(f(y),U)$. Moreover the above cover is said to be an \'etale cover if $\hat Q(y)=\hat P(f(y))$ for all $y\in \widehat Y$ of codimension at least one. The category of \'etale covers of the pair $(\widehat X,\hat P)$ will be denoted by $\Cov(\widehat X,\hat P)$.
\end{df}


\subsection{The functor $\Cov(\widehat{X}_Y,\hat P)$ to $\Cov(Y,P_{|Y})$}

The pullback along the natural morphism $Y\to \widehat X_Y$ defines a functor from $\Cov(\widehat X_Y,\hat P)\to \Cov(Y,P_{|Y})$ (follows in the same way as Proposition \ref{etale-pullback-hypersurface}). We shall show that this functor is an equivalence of category when the following hypothesis holds:

\begin{hypothesis}\label{nonsplit}
  The formal orbifold $(X,P)$ is geometric, irreducible, smooth and projective over $k$; $Y$ is a smooth irreducible hypersurface of $X$ not contained in $\BL(P)$ (with the ideal sheaf $I_Y$); and for any $y\in \BL(\hat P)$ and $U\subset X$ an affine connected neighbourhood of $y$, the ideal $I_Y(U)R(U,y)$ is a prime ideal of $R(U,y)$ where $R(U,y)$ is the integral closure of $\widehat{\cO_{X}(U)}^y$ in $P(y,U)$.
\end{hypothesis}

%

\begin{rmk}
  Note that under the Hypothesis \ref{nonsplit} on $(X,P)$ and $Y$, for $y\in Y$ of codimension at least one $P_{|Y}(y)$ is the fraction field of $R(U,y)/(I_Y(U))$ where $R(U,y)$ is the integral closure of $\widehat{\cO_X(U)}^y$ in $P(y,U)$. The hypothesis can be thought of as a variant of Hilbert irreducibility theorem as well in local setup.
\end{rmk}

\begin{ex}
  Let $X=\PP^n_k$ with homogeneous coordinates $x_0,\ldots x_n$, $f:Z\to X$ be the normalization of $X$ in $k(x_1/x_0,\ldots x_n/x_0)[z]/(z^p-z-x_n/x_0)$ and the branch data $P=B_f$. Let $1\le i \le n-1$ be an integer, $Y$ be the hypersurface in $\PP^n$ defined by $x_i=0$. Then it is easy to check that $(X,P)$ and $Y$ satisfy Hypothesis \ref{nonsplit}
\end{ex}

Let the formal orbifold $(X,P)$ and the hypersurface $Y$ satisfy Hypothesis \ref{nonsplit}. For a closed point $y\in \widehat X_Y$ and an affine neighbourhood $\widehat U_y \subset \widehat X_Y$ of $y$, define $A(\widehat U_y)$ to be the integral closure of $\cO_{\widehat X_Y}(\widehat U_y)$ in $P(\eta_Y,U_y)$ where $U_y\subset X$ is an affine neighbourhood of $y$ viewed as a point in $X$ whose completion along $Y\cap U_y$ is $\widehat U_y$. Let $\scrA_{\widehat U_y}$ be the sheaf of algebras on $\widehat U_y$ associated to $A(\widehat U_y)$. Note that $\scrU=\{\widehat U_y:y\in Y \text{ closed point} \}$ is an open cover of $\widehat X_Y$ and $P$ being branch data there are natural isomorphisms between $\scrA_{U_1}(U_1\cap U_2)$ and  $\scrA_{U_2}(U_1\cap U_2)$  for all $U_1, U_2 \in \scrU$ and they behave well with triple intersections. Hence there is a sheaf of algebras $\scrA$ on $\widehat X_Y$ whose restriction to $\widehat U_y$ is $\scrA_{\widehat U_y}$. Note that $\scrA$ is a sheaf of finite coherent $\cO_{\widehat X_Y}$-algebras.

\begin{df}\label{df:normalization}
  Let $\widehat Z=\spf(\scrA)$ be the formal scheme associated to $\scrA$ and $c:\widehat Z\to \widehat X_Y$ be the structure morphism. We shall call $\widehat Z$ to be the normalization of $\widehat X_Y$ in $P$. Note that $\widehat Z$ is normal and $c$ is a finite morphism.
\end{df}

\begin{pro}\label{P-cover}
  Let $(X,P)$ and $Y$ satisfy the Hypothesis \ref{nonsplit}. Let $F$ be the residue field of the local ring $\overline{\widehat \cO_{X,\eta_Y}}^{P(\eta_Y)}$ where $\eta_Y$ is the generic point of $Y$. Then $F/k(Y)$ is a Galois extension with $\Gal(F/k(Y))=\Gal(P(\eta_Y)/\cK_{X,\eta_Y})$. Let $\widehat Z$ be the normalization of $\widehat X_Y$ in $P$ (in the sense of the above definition) and $c_0:Z_0\to Y$ be the normalized pull-back of $c:\widehat Z\to \widehat X_Y$ along $Y\to \widehat X_Y$ then $Z_0$ is the normalization of $Y$ in $F$. Moreover, the morphisms $(\widehat Z,O)\to (\widehat X_Y,\hat P)$ and $(Z_0,O)\to (Y,P_{|Y})$ are \'etale morphisms.
\end{pro}

\begin{proof}
  Note that $P(\eta_Y)/\cK_{X,\eta_Y}$ is a Galois extension and the point $\eta_Y$ is not in $\BL(P)$ so the residue field extension $F/k(Y)$ is Galois with an epimorphism from $\Gal(P(\eta_Y)/\cK_{X,\eta_Y})\to \Gal(F/k(Y))$. Let $I_Y$ be the sheaf of ideals on $X$ defining $Y$ and $I_{Y,\eta_Y}$ be the prime ideal of $\widehat{\cO}_{X,\eta_Y}$generated by the stalk of $I_Y$ at $\eta_Y$. Note that the Hypothesis \ref{nonsplit} ensures that there is only one prime ideal lying above $I_{Y,\eta_Y}$ in $\overline{\widehat \cO_{X,\eta_Y}}^{P(\eta_Y)}$ and $I_{Y,\eta_Y}$ is unbranched in the extension $\overline{\widehat \cO_{X,\eta_Y}}^{P(\eta_Y)}/\widehat \cO_{X,\eta_Y}$. Hence the degree of the residue extension $F/k(Y)$ is same as that of $P(\eta_Y)/\cK_{X,\eta_Y}$. Hence the two Galois groups are isomorphic.
  
  Note that $\widehat Z\to \widehat X_Y$ is generically \'etale and $\eta_Y$ is unramified and non-split in this cover. Hence the closed fiber, i.e., the pull-back of the formal scheme $\widehat Z$ along $Y\to \widehat X_Y$ is an integral scheme and its function field is $F$. Hence $Z_0$ is the normalization of the closed fiber of $\widehat Z$. Since the closed fiber of $Z$ is a finite cover of $Y$, $Z_0$ is the normalization of $Y$ in $F$.
  
  For a point $z$ of $\widehat Z$ of codimesion at least one in the closed fiber, let $x$ denote its image in $\widehat X_Y$. Then $x$ is of codimension at least one in $Y$. Note that $\cK_{Z,z}=\cK_{\widehat Z,z}$ is the fraction field of the completion of $\overline{\cO_{\widehat X_Y,x}}^{P(\eta_Y)}$ along the prime ideal $I_z$ corresponding to the point $z$. But this ring is same as the integral closure of $\widehat\cO_{\widehat X_Y,x}$ in $P(\eta_Y)\cK_{X,x}$. Hence $\cK_{Z,z}=P(\eta_Y)\cK_{X,x}=P(x)$. This proves $(Z,O)\to (\widehat X_Y,P)$ is \'etale. 
  
  
  Let $z_0\in Z_0$ be a point of codimension at least one and $y\in Y$ be its image. Then $\cK_{Z_0,z_0}=\cK_{Y,y}F$ since $F$ is the function field of $Z_0$. But $F=\overline{\widehat \cO_{X,\eta_Y}}^{P(\eta_Y)}/(I_{Y,\eta_Y})$ as the Hypothesis \ref{nonsplit} implies that $I_{Y,\eta_Y}$ generates the maximal ideal of $\overline{\widehat \cO_{X,\eta_Y}}^{P(\eta_Y)}$. Now,
  \begin{align*}
    \cK_{Z_0,z_0}&=\cK_{Y,y}\overline{\widehat \cO_{X,\eta_Y}}^{P(\eta_Y)}/(I_{Y,\eta_Y}) \\
	       &=\cK_{Y,y}\widehat \cO_{\widehat Z,\eta_Y}/(I_{Y,\eta_Y})\\
	       &=\cK_{Y,y}\QF(\cO_{\widehat V}/(I_Y)) \text{ where $\widehat V\subset \widehat Z$ is an affine open neighbourhood of $\eta_Y$}\\
	       &=\QF(\widehat \cO_{\widehat V, z}/(I_{Y,z}))\text{ where $z$ is the image of $z_0$ in $\widehat Z$}\\
	       &=\QF(\overline{\widehat \cO_{\widehat X_Y, y}}^{P(y)}/(I_{Y,y}))\\
	       &=P_{|Y}(y)
  \end{align*}
  Hence $(Z_0,O)\to (Y,P_{|Y})$ is \'etale. 
\end{proof}

Now we proceed following the strategy in \cite[Section 4.3]{Grothendieck-Murre} to show that the functor $\Cov(\widehat{X}_Y,\hat P) \to \Cov(Y,P_{|Y})$ is an equivalence of category. 

\begin{pro}
  Under the Hypothesis \ref{nonsplit}, the functor $\Cov(\widehat{X}_Y,\hat P) \to \Cov(Y,P_{|Y})$ is faithful.
\end{pro}

\begin{proof}
  Let $f,g:\scrU\to\scrV$ be morphisms in $\Cov(\widehat X_Y,\hat P)$ and $u:\scrU\to (\widehat X_Y,\hat P)$ and $v:\scrV\to (\widehat X_Y,\hat P)$ be \'etale morphism. Assume $f_0=g_0:\scrU_0\to\scrV_0$ where $\scrU_0$ and $\scrV_0$ are the normalized pullback of $\scrU$ and $\scrV$ along $Y\to \widehat X_Y$. We wish to show $f=g$. Let $f^*$ and $g^*$ from $\scrU^*\to \scrV^*$ be the pullback of $f$ and $g$ respectively along $(Z,O)\to (\widehat X_Y,\hat P)$ of the above proposition. Note that the branch data on $\scrU^*$ and $\scrV^*$ are trivial (as $(Z,O)$ has trivial branch data and $u$, $v$ are \'etale). Also $\scrU^*\to (Z,O)$ and $\scrV^*\to (Z,O)$ are \'etale morphisms. Let $(f_0)^*=(g_0)^*$ be the morphism $\scrU_0^*\to \scrV_0^*$ where $\scrU_0^*$ and $\scrV_0^*$ are the normalized pull-backs of $\scrU_0$ and $\scrV_0$ along $Z_0\to Y$. Observe that the normalization of the closed fibers of $\scrU^*$ and $\scrV^*$ are same as $\scrU_0^*$ and $\scrV_0^*$ respectively. Also note that $(f^*)_0=(f_0)^*=(g_0)^*=(g^*)_0$ where $(f^*)_0$ and $(g^*)_0$ are morphisms between the normalization of the closed fibers of $\scrU^*$ and $\scrV^*$. Hence $f^*=g^*$ by \cite[\'etale case]{Grothendieck-Murre}. Now away from the branch locus $Z\to \widehat X_Y$ is \'etale, hence by applying flat descent on this locus there exists an open dense $W \subset \widehat X_Y$ such that $f=g$ when restricted to $u^{-1}(W)$.  But $u^{-1}(W)$ is dense in $\scrU$, hence $f=g$.  
\end{proof}

\begin{pro}
 Under the Hypothesis \ref{nonsplit}, the functor $\Cov(\widehat{X}_Y,\hat P) \to \Cov(Y,P_{|Y})$ is fully faithful.
\end{pro}

\begin{proof}
  Again we use the argument similar to \cite[4.3.6]{Grothendieck-Murre}. Let $Z\to \widehat X_Y$ and $Z_0\to Y$ be the covers in Proposition \ref{P-cover} and $G$ denote their common Galois group. Let $u:\scrU\to (\widehat X_Y,\hat P)$ and $v:\scrV\to (\widehat X_Y,\hat P)$ be \'etale covers and $g_0:\scrU_0\to \scrV_0$ be a morphism in $\Cov(Y,P_{|Y})$. Let $g_0^*:\scrU_0^*\to \scrV_0^*$ be the morphism obtained by pulling back $g_0$ to $Z_0$. Now $g_0^*$ is \'etale and the branch data on $\scrU_0^*$ and $\scrV_0^*$ are trivial. Rest of the proof is same as \cite[4.3.6]{Grothendieck-Murre} with $\mu_n$ replaced by $G$. 
\end{proof}

\begin{thm}\label{GM-wild}
 Under the Hypothesis \ref{nonsplit}, the functor $\Cov(\widehat{X}_Y,\hat P) \to \Cov(Y,P_{|Y})$ is an equivalence.
\end{thm}

\begin{proof}
 Here also the proof is same as in \cite[4.3.7]{Grothendieck-Murre} with appropriate modifications, i.e. the Kummer covers being replaced by the covers $Z\to \widehat X_Y$ and $Z_0\to Y$ of Proposition \ref{P-cover}.
\end{proof}

\subsection{The functor $\Cov(X,P)$ to $\Cov(\widehat X_Y, \hat P)$}

We need the following lemma.

\begin{lemma}\label{depth}
 Let $R$ be a ring and $M$, $N$ be $R$-modules, let $r_1,r_2$ be a regular sequence on $N$ then $r_1, r_2$ is a regular sequence on $\Hom_R(M,N)$. In particular, if depth of $N$ is at least 2 then depth of $\Hom_R(M,N)$ is also at least 2.
\end{lemma}

\begin{proof}
 Let $f\in \Hom_R(M,N)$ and $r_1f=0$ then $r_1f(m)=0$ for all $m\in M$. But $r_1$ is regular on $N$ implies $f(m)=0$ for all $m\in M$, i.e. $f=0$. Hence $r_1$ is a nonzero divisor on $\Hom_R(M,N)$. 
 
 Let $\bar g\in \Hom_R(M,N)/r_1\Hom_R(M,N)$ be the image of $g\in \Hom_R(M,N)$. We need to show that if $r_2\bar g=0$ then $g\in r_1\Hom_R(M,N)$. Note that $r_2\bar g=\overline{r_2 g}=0$ implies $r_2g\in r_1\Hom_R(M,N)$. Hence there exists $f\in \Hom_R(M,N)$ such that $r_2g(m)=r_1f(m)$ for all $m\in M$. Hence $g(m)\in r_1N$ for all $m\in M$. But $r_1, r_2$ is a regular sequence on $N$, so $g(m)\in r_1N$ for all $m\in M$. Hence $g\in r_1\Hom_R(M,N)$.
\end{proof}

The pullback along the natural morphism $\widehat X_Y \to X$ defines a functor from $\Cov(X,P) \to \Cov(\widehat X_Y,\hat P)$. 

\begin{pro}\label{formal-fullyfaithful}
 Let $(X,P)$ be a projective formal orbifold over $k$ of dimension at least two, $Y$ be a normal hypersurface of $X$ which intersects $\BL(P)$ transversally and the divisor $[Y]$ is a very ample divisor on $X$. Then the functor $\Cov(X,P)\to \Cov(\widehat X_Y,\hat P)$ is fully faithful.
\end{pro}

\begin{proof}
  Let $u:\scrZ \to (X,P)$ and $v:\scrW\to (X,P)$ be in $\Cov(X,P)$. It is enough to show that $\Hom_X(Z,W)\cong \Hom_{\widehat X_Y}(Z_Y,W_Y)$ where $Z_Y$ and $W_Y$ are the pullback of $u$ and $v$ to $\widehat X_Y$. Note that $Z_Y$ and $W_Y$ are normal. Note that $\Hom_X(Z,W)$ are same as homomorphisms between sheaf of $\cO_X$-algebras $v_*\cO_W\to u_*\cO_Z$ as $u$, $v$ and every $X$-morphism from $Z$ to $W$ are finite (and hence affine). Since $Z$ and $W$ are normal of dimension at least 2, $u_*O_Z$ has depth at least 2 and hence by Lemma \ref{depth} coherent sheaf of $\cO_X$-modules $\underline{\Hom}_{\cO_X}(v_*\cO_W,u_*\cO_Z)$ has depth at least 2 (here we consider all $\cO_X$-module homomorphism and not just algebra homomorphism).
  Using \cite[Chapter XII, Corollary 2.2]{SGA2}, we obtain that $\Hom_{\cO_X}(v_*\cO_W,u_*\cO_Z)\cong \Hom_{\cO_{\widehat X_Y}}(\widehat {v_*\cO_W},\widehat{u_*\cO_Z})$ but the right side equals $\Hom_{\cO_{\widehat X_Y}}(v_*\cO_{\widehat W_Y},u_*\cO_{\widehat Z_Y})$. Also since $\widehat X_Y \to X$ is flat $\Hom$ behaves well with tensor product. So the bijection takes an algebra homomorphism between $v_*\cO_W$ and $u_*\cO_Z$ to an algebra homomorphism between $v_*\cO_{\widehat W_Y}$ and $u_*\cO_{\widehat Z_Y}$. Hence the functor is fully faithful.
\end{proof}

In fact like in \cite[Chapter XII]{SGA2}, even the following is true.

\begin{pro}\label{nbh-faithful}
  Let $U$ be an open neighborhood of $Y$. Then the functors $\Cov(X,P)\to \Cov(U, P_{|U})$ and $\Cov(U, P_{|U})\to \Cov(\widehat X_Y,\hat P)$ are fully faithful. 
\end{pro}

\begin{proof}
  Note that the composition of these two functors is the functor in the above proposition. We use the notation of the above proof. Since $\underline{\Hom}_{\cO_X}(v_*\cO_W,u_*\cO_Z)$ has depth at least 2 (Lemma \ref{depth}), it is torsion free. Hence the homomorphisms $\Hom_{\cO_X}(v_*\cO_W,u_*\cO_Z)\to \Hom_{\cO_U}(v_*\cO_{W_U},u_*\cO_{Z_U})$ and $\Hom_{\cO_U}(v_*\cO_{W_U},u_*\cO_{Z_U})\to \Hom_{\cO_{\widehat X_Y}}(v_*\cO_{\widehat W_Y},u_*\cO_{\widehat Z_Y})$ are injective (for the injectivity of the second map use the argument in \cite[Chapter XII, Corollary 2.4]{SGA2}). By the above proposition the composition is an isomorphism hence the two maps are isomorphisms. 
\end{proof}

As a consequence of Theorem \ref{GM-wild} and Proposition \ref{formal-fullyfaithful} we get the following corollary. 

\begin{cor}\label{Lefschetz-thm}
  Let $(X,P)$ be a smooth irreducible projective formal orbifold over a perfect field $k$ of dimension at least two. The functor from $\Cov(X,P)\to \Cov(Y,P_{|Y})$ is fully faithful if $Y$ satisfies Hypothesis \ref{nonsplit}. In particular the map natural map $\pi_1(Y,P_{|Y})\to \pi_1(X,P)$ is an epimorphism. 
\end{cor}

 \section{Locally constant $l$-adic sheaf}\label{sec:l-adic}
 
 Let $X$ be a normal geometrically connected variety over a field $k$ of characteristic $p$. Recall that given a finite locally constant sheaf $\scrF$ on $X$ with stalk $A$ 
 it corresponds to a representation $\rho_{\scrF}:\pi_1(X,x)\to \Aut(A)$ where $x$ is a geometric point of $X$. The converse also holds. This result and its lisse $l$-adic version hold for geometric formal orbifolds as well.
 
 \begin{df}
   Let $l\ne p$ be a prime number. An $l$-adic sheaf is a locally constant sheaf of $R$-modules where $R$ is the integral closure of $\Z_l$ in a finite extension $K$ of $\Q_l$. Note that $R$ is a complete DVR and let $\scrm$ denote its maximal ideal. A lisse $l$-adic sheaf of $R$-modules is a compatible system of locally constant sheaves $(\scrF_n)_{n\ge 0}$ where $\scrF_n$ is a locally constant $R_n=R/\scrm^{n+1}$-module and $\scrF_{n+1}\otimes_{R_{n+1}}R_n\cong \scrF_n$. The morphisms in this category are defined to be compatible system of morphisms. For more details see \cite[Chapter 1, Section 12]{FK}. The category of lisse $K$-sheaf consist of lisse $l$-adic sheaf of $R$-modules as objects and the morphisms are given by  $Hom_K(\scrF,\scrF'):=\Hom(\scrF',\scrF')\otimes_R K$. 
 \end{df}

 \begin{pro}\label{l-adicsheaf.Galois-repr}
 Let $\scrF$ be a finite locally constant sheaf with stalk $A$ on a geometric formal orbifold $\scrX=(X,P)$. Let $X^o=X\setminus \BL(P)$. The restriction $\scrF^o=\scrF_{|X^o}$ is a locally constant sheaf and hence $\scrF^o$ correspond to a representation $\rho:\pi_1(X^o,x)\to Aut(\scrF_x)$ where $x\in X^o$ is a geometric point. Then $\rho$ factors through $\pi_1(\scrX,x)$ to give a continuous representation $\bar \rho:\pi_1(\scrX,x)\to \Aut(\scrF_x)$. Conversely given a continuous representation $\bar \rho:\pi_1(\scrX,x)\to \Aut(A)$ it induces a finite locally constant sheaf $\scrF$ on $\scrX$ with stalk $A$. And the two functors are inverse to each other. 
\end{pro}

\begin{proof}
  Since $\scrX=(X,P)$ is a geometric formal orbifold and $X$ is geometrically connected there exists a finite Galois \'etale cover of formal orbifolds $h:(Z,O)\to \scrX$, with $Z$ geometrically connected. 
  Since $\scrF$ is a locally constant sheaf on $\scrX$ so is $h^*\scrF$. The \'etale site on $(Z,O)$ is the usual \'etale site on $Z$. Hence there exists an \'etale Galois cover $g:Y\to Z$ with $Y$ connected such that the pull back $g^*h^*\scrF$ is constant. Let $f=h\circ g$ and passing to Galois closure we may assume $f$ is a Galois cover. Let $Y^o=f^{-1}(X^o)$. The restriction of $f^*\scrF$ to $Y^o$ is constant. Hence $\rho$ factors through $\pi_1(X^o)/\pi_1(Y^o)$ which is same as $\pi_1(\scrX)/\pi_1(Y)$. Hence $\rho$ factors through $\pi_1(\scrX)$. 
  
  Conversely, let $\bar \rho:\pi_1(\scrX,x)\to Aut(A)$ be a continuous representation. Since $A$ is finite, $\bar \rho$ factors through a finite quotient of $\pi_1(\scrX)$ say $G$. This finite quotient correspond to a $G$-Galois \'etale cover $f:(Y,Q)\to \scrX$ and we may assume $Q$ is the trivial branch data. Consider the constant sheaf $A_Y$ on $Y$. The group homomorphism $G\to Aut(A)$ makes $A_Y$ into a $G$-sheaf. In other words the action of $G$ on $Y$ lifts to $A_Y$ via the group homomorphism $G\to A_Y$. Then $f_*A_Y$ is a sheaf on $\scrX$ with a $G$-action. Set $\scrF:=(f_*A_Y)^G$. Then $\scrF$ is a locally constant sheaf on $\scrX$. 
  
  The two functors are inverse to each other can be seen by using the corresponding result in the variety case. 
\end{proof}

\begin{cor}
  The category of lisse $l$-adic sheaf of $K$-modules on $\scrX$ is equivalent to the category of continuous representations $\rho:\pi_1(\scrX)\to \GL_N(K)$ for any $K/\Q_l$ a finite field extension. 
\end{cor}

\begin{proof}
 In view of the above proposition, the proof is same as the case of normal varieties.
\end{proof}

We collect some definitions on tameness and bounded ramification of $l$-adic sheaves from \cite{Drinfeld}, \cite{Esnault} and \cite{Esnault-Kerz.Deligne}. 
\begin{df}
  Let $X$ be a normal geometrically connected variety over $k$ and $\scrF$ a locally constant sheaf on $X$.
  \begin{enumerate} 
   \item  If $\dim(X)=1$ then  $\scrF$ is called tame if $\rho_{\scrF}$ factors through the tame quotient of the fundamental group $\pi_1^t(X,x)$. 
 
   \item For $X$ of arbitrary dimension, $\scrF$ is tame if for any geometrically irreducible curve $C$ in $X$ with normalization $\tilde C$, the pull back of $\scrF$ to $\tilde C$ is tame.


   \item A lisse $l$-adic sheaf $\scrF=(\scrF_n)$ on $X$ is called tame if for all $n$, $\scrF_n$ are tame on $X$. 

   \item A lisse $l$-adic sheaf $\scrF$ on $X$ is said to have ramification bounded by $f:Y\to X$ a finite dominant morphism of normal varieties if $f^*\scrF$ is tame on $Y$.
  \end{enumerate} 
\end{df}

The following notions of tameness of a branch data are defined in analogy with  \cite{Kerz-Schmidt.tame}.
\begin{df}
  Let $P$ be a branch data on a normal variety $X$.
  \begin{enumerate}
   \item For a point $x\in X$ of codimension at least one, $P$ is called numerically-tame at $x$ if the inertia group $I(P(x)/\cK_{X,x})$ is of order prime to $p$. The branch data $P$ is called numerically-tame if $P$ is numerically tame for all points $x\in X$ of codimension at least one. It is enough to check numerically-tame for all closed points. 
   \item For a closed point $x\in X$, $P$ is called curve-tame at $x$ if the following holds: For any codimension one prime ideal $\scrp$ in $\widehat \cO_{X,x}$ not in $\BL(P)$ and for any codimension one prime ideal $\scrq$ in the integral closure $\cO_{P(x)}$ of $\widehat \cO_{X,x}$ in $P(x)$ lying above $\scrp$, let $R$ and $S$ be the normalization of $\widehat \cO_{X,x}/\scrp$ and $\cO_{P(x)}/\scrq$ respectively. Then $S/R$ is at most tamely ramified extension. The branch data $P$ is called curve-tame if it is curve-tame for all closed points $x\in X$.    
  \end{enumerate}
\end{df}

\begin{pro}
 Let $X$ be a proper normal variety over $k$ and $f:Y\to X$ be a Galois cover which is \'etale over $X^o$ an open subset of regular locus of $X$. The branch data $B_f$ is curve-tame iff $f$ is curve-tame.  
\end{pro}

\begin{proof}
  Assume $f$ is curve-tame. Let $x\in X$ be a closed point. Let $\scrp$ be a codimension one prime ideal of $\widehat \cO_{X,x}$ not in $\BL(B_f)$ and $\scrq$ be a codimension one prime ideal in the integral closure $\cO_{B_f(x)}$ of $\cO_{X,x}$ in $B_f(x)$ lying above $\scrp$. Then $\scrp$ defines an integral curve $C$ in $X$ passing through $x$ which intersects $X^o$ and $C$ is unibranched at $x$. Let $\tilde C$ be the normalization of $C$ and $\tilde x$ be the point in $\tilde C$ lying above $x\in C$. Let $Y_C= \widetilde{ Y\times_X \tilde C}$ be the normalized fiber product. Since $f$ is curve-tame, $Y_C \to \tilde C$ is at most tamely ramified. Note that there exists $y\in Y$ lying above $x$ such that $\widehat\cO_{Y,y}=\cO_{B_f(x)}$. The  normalization $S$ of $\widehat \cO_{Y,y}/\scrq$ is same as $\widehat \cO_{D,z}$ where $D$ is a connected component of $Y_C$ and $z$ is a point on $Y_C$ lying above $(y,\tilde x)$ in the fiber product. Since $Y_C \to \tilde C$ is at most tamely ramified, the extension $\widehat \cO_{D,z}/\widehat\cO_{\tilde C,\tilde x}$ is at most tamely ramified. But $\widehat\cO_{\tilde C,\tilde x}$ is the normalization of $\widehat \cO_{X,x}/\scrp$. Hence $B_f$ is curve tame at $x$.
  
  The converse is also a translation between algebra and geometry and follows after noting that the tameness of ramification does not change when passed to completion.
\end{proof}

The following result is a direct consequence of \cite[Theorem 5.3, Theorem 5.4]{Kerz-Schmidt.tame}.
\begin{pro}\label{tame.covers.equivalence}
 Let $(X,P)$ be a proper formal orbifold with $P$ a numerically-tame branch data and $X^o=X\setminus \BL(P)$ regular. Let $f:(Y,O)\to(X,P)$ be an \'etale Galois cover. Then $f:Y\to X$ is curve-tame. In other words $P$ is curve tame.
 The converse hold if $X$ is regular and $\BL(P)$ is a normal crossing divisor.
\end{pro}

\begin{proof}
  The cover $f:(Y,O)\to (X,P)$ is an \'etale cover implies that $f:Y\to X$ is \'etale over $X^o$ and numerically tamely ramified cover w.r.t to the normal compactification $X$. Now the result follows from \cite[Theorem 5.3, Theorem 5.4]{Kerz-Schmidt.tame}
\end{proof}

\begin{pro}
 Let $X$ be a normal connected proper variety and $X^o$ be an open subset of the regular locus of $X$. Let $f:Y\to X$ be a Galois cover \'etale over $X^o$. Any finite locally constant $l$-adic sheaf $\scrF$ on a geometric formal orbifold $(X,P)$ where $P\ge B_f$ with $\BL(P)\cap X^o=\emptyset$ and $P(x)/B_f(x)$ at most curve tamely ramified for all closed points $x\in X$ restricts to a locally constant finite $l$-adic sheaf on $X^o$ whose ramification is bounded by $f$. 
\end{pro}

\begin{proof}
 Let $\scrF^o=\scrF|_{X^o}$. Let $g:(Z,O)\to (X,P)$ be an \'etale cover such that $g^*\scrF$ is a constant sheaf. Let $(W,O)$ be the fiber product of $g$ and $f:(Y,f^*P)\to(X,P)$ and $h:(W,O)\to (Y,f^*P)$ be the projection map. Note that $h$ is \'etale and $f^*P$ is a curve-tame branch data since $P(x)/B_f(x)$ is curve-tamely ramified for all $x\in X$. So $h:W\to Y$ is a curve-tamely ramified cover. Now $f^*\scrF^o$ is a locally constant sheaf on $Y^o=f^{-1}(X^o)$ and $h^*f^*\scrF^o=g^*\scrF^o$ is a constant sheaf. Hence $f^*\scrF^o$ is tame $l$-adic sheaf. 
\end{proof}

The converse also holds.

 \begin{pro}
  Let $X$ be a normal connected proper variety and $X^o$ be an open subset of the regular locus of $X$. Let $\scrF^o$ be a locally constant finite $l$-adic sheaf on $X^o$ whose ramification is bounded by a Galois cover $f:Y\to X$ which is \'etale over $X^o$. Let $B_f$ be the branch data on $X$ associated to $f$. Then $\scrF^o$ extends to a finite locally constant $l$-adic sheaf $\scrF$ on $(X,P)$ for some geometric branch data $P\ge B_f$ with $\BL(P)\cap X^o=\emptyset$ and $P(x)/B_f(x)$ at most curve-tamely ramified for all closed points $x\in X$. 
 \end{pro}
 
 \begin{proof}
  By Proposition \ref{l-adicsheaf.Galois-repr} a finite locally constant $l$-adic sheaf $\scrF^o$ on $X^o$ correspond to a continuous representation $\rho:\pi_1(X^o,x)\to \Aut(\scrF^o_x)$. Since ramification of $\scrF^o$ is bounded by $f$, $f^*\scrF^o$ is tame finite locally constant sheaf. So there exists a finite curve-tamely ramified Galois cover $h':Y'\to Y$ \'etale over $X^o$ such that $h^*f^*\scrF^o$ is a constant sheaf. Let $g:Z\to X$ be the Galois closure of $f\circ h': Y'\to X$. Then $g$ is \'etale over $X^o$ and the natural map $h:Z\to Y$ is curve tamely ramified (being compositum of covers isomorphic to $h':Y'\to Y$ which which is curve-tamely ramified). Let $P=B_g$ then $g:(Z,O)\to (X,P)$ is \'etale, $P$ is a geometric branch data, $\BL(P)\cap X^o=\emptyset$, $P\ge B_f$ and $\rho$ factors through $\pi_1((X,P),x)$. Let $\scrF$ be the extension of $\scrF^o$ to $(X,P)$ (obtained from Proposition \ref{l-adicsheaf.Galois-repr}). The extension $\cK_{Z,z}/\cK_{Y,h(z)}$ is at most curve-tamely ramified for all closed points $z\in Z$ since $h$ is curve-tamely ramified. But $\cK_{Z,z}=P(g(z))$ and $\cK_{Y,h(z)}=B_f(f(h(z)))=B_f(g(z))$, so $P(x)/B_f(x)$ is at most curve-tamely ramified for all closed points $x\in X$. 
 \end{proof}

%

 \begin{cor} \label{lisse-bounded-ramification}
   Let $\scrF^o$ be a lisse $l$-adic sheaf on $X^o$. Let $X$ be a normal compactification of $X^o$. For $i=1,2$ let $f_i:Y_i\to X$ be Galois covers of normal varieties \'etale over $X^o$ such that $Q_1B_{f_1}=Q_2B_{f_2}$ where $Q_1$ and $Q_2$ are curve-tame branch data on $X$ with branch locus disjoint with $X^o$. Then the ramification of $\scrF^o$ is bounded by $f_1$ iff it is bounded by $f_2$.
 \end{cor}

 \begin{proof}
   Let $\scrF^o=(\scrF^o_n)_{n\ge 0}$ where $\scrF^o_n$ are compatible finite locally constant $l$-adic sheaves. For $i=1,2$, by the above proposition the ramification of $\scrF^o_n$ is bounded by $f_i$ iff $\scrF^o_n$ extends to a locally constant sheaf on $(X,P)$ for some $P\ge B_{f_i}$ with $\BL(P)\cap X^o=\empty$ and $P(x)/B_{f_i}(x)$ at most tamely ramified for all $x\in X$. Since $Q_1$ and $Q_2$ are tame branch data on $X$, for $i=1,2$, $P(x)Q_i(x)/B_{f_i}(x)Q_i(x)$ is at most tamely ramified iff $P(x)/B_{f_i}(x)$ is at most tamely ramified for all $x\in X$. Since $Q_1B_{f_1}=Q_2B_{f_2}$ we obtain that $\scrF^o_n$ is bounded by $f_1$ iff it is bounded by $f_2$.  
 \end{proof}
 
 In particular the above means that the local property at the boundary of an \'etale morphism $f:Y^o\to X^o$ decides whether an $l$-adic lisse sheaf on $X^o$ has ramification bounded by $f$.  

%
%

\end{document}